\documentclass[12pt,a4paper]{amsart}
 \usepackage{ifthen}\newboolean{1014} \setboolean{1014}{true}
\usepackage{amsmath,amssymb}

\newtheorem{theorem}{Theorem}[section]
\newtheorem{example}[theorem]{Example}
\newtheorem{lemma}[theorem]{Lemma}
\newtheorem{proposition}[theorem]{Proposition}
 
\newtheorem{remark}[theorem]{Remark}

\newtheorem{problem}{Problem}

\def\Ga{{\Gamma}}

\def\de{\delta}
\def\al{\alpha}
\def\be{\beta}
\def\ga{\gamma}

\newcommand{\GL}{\mathop{\mathrm{GL}}}
\newcommand{\F}{\mathbb{F}}

\def\Sym{{\rm Sym}\,}

\def\Sz{{\rm Sz}\,}
\def\Ree{{\rm Ree}\,}

\def\SL{{\rm SL}\,} 
\def\PGaL{{\rm P\Gamma L}\,} 
\def\Sp{{\rm Sp}\,} 
\def\PSL{{\rm PSL}}
\def\PGL{{\rm PGL}}
\def\AG{{\rm AG}}
\def\AGL{{\rm AGL}}
\def\GaL{{\rm \Gamma L}}
\def\PG{{\rm PG}}
\def\PSU{{\rm PSU}\,}
\def\AGaL{{\rm A\Gamma L}}
\def\PGaU{{\rm P\Gamma U}}

\def\SU{{\rm SU}}

\def\GU{{\rm GU}}
\def\Aut{{\rm Aut}\,}
\def\Out{{\rm Out}\,}
\def\Stab{{\rm Stab}}

\def\ra{\rangle}
\def\la{\langle}

\def\si{\sigma}
\def\lam{\lambda}

\def\calL{\mathcal{L}}
\def\Xv{\binom{\V}{k}}
\def\ov{\overline}

\def\V{\mathcal{V}}
\def\U{\mathcal{U}}

\def\e{\mathsf{e}}
\def\v{\mathsf{v}}
\def\w{\mathsf{w}}
\def\u{\mathsf{u}}
\def\t{\mathsf{t}}
\def\x{\mathsf{x}}
\def\y{\mathsf{y}}

\title{Neighbour-transitive codes in Johnson graphs} 
\author{Robert A. Liebler and Cheryl E. Praeger}
\date{December 2012}

\begin{document}
\begin{abstract}
The Johnson graph $J(v,k)$ has, as vertices, the $k$-sub\-sets of a $v$-set $\V$
and as edges the pairs of $k$-subsets with intersection of size $k-1$.
We introduce the notion of a neighbour-transitive code in $J(v,k)$.
This is a vertex subset $\Gamma$
such that the subgroup $G$ of graph automorphisms leaving $\Gamma$ 
invariant is transitive on  both the set $\Gamma$ of `codewords'
and also the set of `neighbours' of $\Gamma$, which are the non-codewords
joined by an edge to some codeword. We classify all examples 
where the group $G$ is a subgroup of the symmetric group $\Sym(\V)$ 
and is intransitive or imprimitive on the 
underlying $v$-set $\V$. In the remaining case where $G\leq\Sym(\V)$
and $G$ is primitive on $\V$, we prove that, provided distinct codewords
are at distance at least $3$, then $G$ is $2$-transitive on $\V$. 
We examine many of the infinite families of finite  $2$-transitive 
permutation groups and construct surprisingly rich families of examples
of neighbour-transitive codes. A major unresolved case remains.


\end{abstract}

\maketitle

\section{Introduction}\label{intro}

In 1973, Philippe Delsarte \cite{Delsarte} introduced the notion of a code 
in a distance regular graph, namely a vertex subset whose elements are the 
codewords and with distance between codewords being the natural distance 
in the graph. In particular he defined a special class of such codes, now 
called completely regular codes, `which enjoy combinatorial (and often 
algebraic) symmetry akin to that observed for perfect codes'\ \cite[page 1]{Martin04}. (Completely regular codes are defined in Subsection~\ref{creg}.)
Disappointingly, not many completely regular codes with good error-correcting properties 
(large distance between distinct codewords) were found and, for such codes in binary Hamming graphs having at 
least three codewords, it has been conjectured that the minimum distance between 
distinct codewords is at most 8 (see \cite[page 2]{BRZ}). In fact
Neumaier~\cite{Neumaier} conjectured that the only completely regular 
code with minimum distance 8 in a binary Hamming graph is the extended binary Golay 
code. Even though Neumaier's conjecture was disproved by 
Borges, Rifa, and Zinoviev~\cite{BRZ2}, there are very few codes known with
these properties.

Delsarte's paper \cite{Delsarte} posed explicitly the question of existence 
of completely regular codes in Johnson graphs, and our focus in this paper
is on a related family of codes in these graphs which contains many completely regular examples. 
Completely regular codes in Johnson graphs have been studied by 
Meyerowitz~\cite{Meyerowitz1,Meyerowitz2} and 
Martin~\cite{Martin94,Martin98}. We relax the stringent 
regularity conditions imposed for complete regularity, and replace 
them with conditions involving only codewords and their immediate neighbours. 
On the other hand, we strengthen the regularity conditions for codewords and their 
neighbours to a local transitivity property. The codes we study are called 
neighbour-transitive codes. We construct surprisingly rich classes of 
examples arising from both combinatorial and geometric structures,
including some famlies with unbounded minimum distance.

% Our analysis deals with most kinds of neighbour transitive codes,
% with minimum distance at least 3, but leaves an important unresolved 
% case associated with the 2-transitive actions of symplectic groups 
% on quadratic forms. We discuss this further below. 

Some but not all of the examples we construct are completely regular, generalising the 
constructions and results in~\cite{Martin94,Martin98,Meyerowitz1,Meyerowitz2}. 
Other constructions raise new  questions about geometric configurations in 
projective and affine spaces, and spaces of binary quadratic forms. The 
last case,  associated with the 2-transitive actions of symplectic groups 
on binary quadratic forms, gives rise to a significant open problem \emph{(see below)}. Our work 
generalises also the as yet unpublished study in \cite{GP} by Godsil 
and the second author of completely transitive codes in Johnson graphs. 
 
\bigskip\noindent
\textbf{Dedication:}  This work began as a joint project almost a decade ago,
by Bob Liebler and me.  Sadly, in July 2009, Bob Liebler died while hiking in California.
I completed the paper alone and I dedicate it to my friend and colleague Bob Liebler.

\subsection{Johnson graphs and neighbour-transitive codes}
The  \emph{Johnson graph} $J(v,k)$, based on a set $\V$ of $v$ elements called 
\emph{points}, is  the graph whose vertex set is the set $\binom{\V}{k}$ of all 
$k$-subsets of $\V$, with edges being the unordered pairs $\{\ga,\ga'\}$ of $k$-subsets
such that $|\ga\cap\ga'|=k-1$. 
Since $J(v,1)$ and $J(v,v-1)$ are both the complete graph on $\V$, we assume that
$2\leq k\leq v-2$. Moreover, since $J(v,k)\cong J(v,v-k)$, we may sometimes, when convenient,
restrict our analysis to the case  $k\leq v/2$. This is discussed further in 
Subsection~\ref{rem-flag-tra}. 

The graph $J(v,k)$ admits the symmetric group $\Sym(\V)$ as a group of automorphisms, 
and if $k\ne v/2$ this is the full automorphism group. If $k=v/2$ then the complementation map 
$\tau$ that sends each $k$-subset $\gamma$ to its complement $\overline\gamma :=\V\setminus\gamma$
is also an automorphism of $J(v,k)$ and the full automorphism group is
$\Sym(\V)\times\la\tau\ra\cong S_v\times Z_2$. This exceptional case is investigated in
\cite{NP2}, and in this paper we consider subgroups of 
automorphisms contained in $\Sym(\V)$.

The codes we study are proper subsets $\Ga\subset\binom{\V}{k}$. %of $k$-subsets of $\V$.
The automorphism group $\Aut(\Ga)$ of such a code 
$\Ga$ is the setwise stabiliser of $\Ga$ in the symmetric group
$\Sym(\V)\cong S_v$ (or in $\Sym(\V)\times\la\tau\ra$ if $k=v/2$). 
By a \emph{neighbour of $\Ga$} we mean a $k$-subset $\ga_1$ of 
$\V$ that is not a codeword but satisfies $|\ga_1\cap\ga|=k-1$ for some 
codeword $\ga\in\Ga$, that is to say, the distance $d(\ga,\ga_1)$ between
$\ga$ and $\ga_1$ in $J(v,k)$ is 1.
By the 
\emph{minimum distance} $\de(\Ga)$ of a code $\Ga$, we mean the least 
distance in $J(v,k)$ between distinct codewords of $\Ga$. Thus 
provided $\de(\Ga)>1$, all vertices adjacent to a codeword are neighbours.
We say that $\Ga$ is 
\emph{code-transitive} if $\Aut(\Ga)$ is transitive on $\Ga$, and 
\emph{neighbour-transitive} if $\Aut(\Ga)$ is transitive on  both
$\Ga$ and the set $\Ga_1$ of neighbours of $\Ga$. 

The concept of neighbour-transitivity for codes in $J(v,k)$ can be placed in a 
broader context by viewing a code $\Ga$ and its neighbour set $\Ga_1$ as an
incidence structure, with incidence between a codeword and a neighbour 
induced from adjacency in $J(v,k)$. (See 
Section~\ref{sect:prelims} for more details.) This incidence
structure, and also the code $\Gamma$, is
called \emph{$G$-incidence-transitive} if $G\leq \Aut(\Ga)$ and $G$ is transitive on 
codeword-neighbour pairs $(\ga,\ga_1)$ with $\ga\in\Ga,\ga_1\in\Ga_1$ and
$d(\ga,\ga_1)=1$. Each incidence transitive code is neighbour transitive 
(by definition), but if  $\de(\Ga)\leq 2$, it is possible for $\Ga$ to be 
neighbour-transitive but not incidence-transitive (see
Example~\ref{ex:de1b}), or for $\Ga$ to be code-transitive but 
not neighbour-transitive, or for $\Aut(\Ga)$ to be transitive on $\Ga_1$ 
but not transitive on $\Ga$, and hence not neighbour-transitive
(see Examples~\ref{ex1} and~\ref{ex:de1}).

\subsection{Results and questions}\label{rq}
Neighbour transitivity may seem a rath\-er restrictive condition. 
However examples range from the 
collection of all $k$-subsets of a fixed subset $\U\subseteq \V$ (Example~\ref{ex-intrans}),
to the block set of the $5-(12,6,1)$ Witt design associated with 
the Mathieu group ${\rm M}_{12}$ \cite[Table 1]{NP}, to the set of lines of
a finite projective space (Example~\ref{ex-lin1}). 
Moreover the examples include the completely-transitive designs
studied in \cite{GP} where transitivity is required not only on the
code $\Ga$ and its neighbour set $\Ga_1$, but also on each subset $\Ga_i$ of 
the distance partition determined by $\Ga$ (see Subsection~\ref{creg}). 
Many of the constructions from \cite{GP} were mentioned in Bill Martin's
papers \cite{Martin94,Martin98} on completely regular designs.

As a broad summary of the results of this paper, together with those of 
\cite{Dur} and \cite{NP}, for the case of minimum distance at least $3$, 
we can announce that: 

\begin{center}
\emph{if $\Gamma\subset J(v,k)$ with $\delta(\Gamma)\geq3$ such that
$G:=\Aut(\Gamma)\cap\Sym(\V)$ is neighbour-transitive on $\Gamma$, then either $\Gamma$ 
is known explicitly, or $G$ is a symplectic group acting $2$-transitively on a set
$\V$ of quadratic forms.} 
\end{center}

Thus a major open problem remains, work on which is proceeding in 
the PhD project of Mark Ioppolo at the University of Western Australia.
(Some examples are known in this case.) 

\begin{problem}\label{prob1}{ 
   Classify the $G$-neighbour-transitive codes $\Ga\subset J(v,k)$, 
where $G=\Sp(2n,2)$ and $v=2^{2n-1}\pm 2^{n-1}$.
}
\end{problem}

Our first result is a complete classification 
(proved in Sections~\ref{sect:intrans} and 
\ref{sect:imprim}) of the
neighbour-transitive codes in $J(v,k)$ for which the 
automorphism group is intransitive, or transitive and imprimitive, on 
the point set $\V$. A transitive group 
$A$ is imprimitive on $\V$ if it leaves invariant a
non-trivial partition of $\V$.

\begin{theorem}\label{notprim}
If $\Ga\subset\binom{\V}{k}$ is neighbour-transitive, 
%where $2\leq k\leq |\V|/2$, 
and if $\Aut(\Ga)\cap\Sym(\V)$ is intransitive on $\V$, or transitive and 
imprimitive on $\V$, then $\Ga$ is one of the codes in Example~{\rm\ref{ex-intrans}, 
\ref{ex:imprim}}, or {\rm\ref{ex:imprim2}}.
\end{theorem}

For a code $\Ga$ and group $G\leq\Aut(\Ga)\cap\Sym(\V)$, we say that $\Ga$ is
$G$-\emph{strongly incidence-transitive} if $G$ is transitive on $\Ga$ and,  
for $\ga\in\Ga$,   $G_\ga$ is 
transitive  on the set of pairs $(\u,\u')$ with $\u\in\ga, \u'\in \V\setminus\ga$.
It is not hard to see that each strongly incidence transitive code
(which by definition is a proper subset of $\binom{\V}{k}$) is 
incidence transitive, and indeed there exist  incidence transitive codes 
$\Ga$ which are not strongly incidence transitive, necessarily with 
$\delta(\Ga)=1$. Examples of such codes are given in Examples~\ref{ex-intrans}
and~\ref{ex:imprim2}, see Lemmas~\ref{lem-intrans} and~\ref{lem:imprimex2},
respectively. The next result Theorem~\ref{flagtra}
%(If a group $H$ acts on two sets $\U$ and $\U'$, then by its product action on 
%$\U\times\U'$ we mean its natural action with $h\in H$ mapping $(\u,\u')$ to $(\u^h,{\u'}^h)$.)
links the notions of incidence-transitivity, 
strong incidence-transitivity and neighbour-transitivity, 
and provides critical information about the case where $G$ is primitive on $\V$.
It is proved in Section~\ref{sect:primitive}.

%For $G\leq\Aut(\Ga)$ we say that $\Ga$ is 
%$G$-neighbour-transitive if $G$ is transitive on both $\Ga$ and $\Ga_1$; 
%similarly for $G$-incidence-transitive, etc.

\begin{theorem}\label{flagtra}
Let $\Ga\subset\binom{\V}{k}$  and $G\leq\Aut(\Ga)\cap\Sym(\V)$, 
where $2\leq k\leq|\V|-2$.  
\begin{enumerate}
\item[(a)] The code  $\Ga$ is $G$-strongly incidence-transitive if and only if 
$\Ga$ is $G$-incidence-transitive and $\de(\Ga)\geq2$.
\item[(b)] If $\de(\Ga)\geq3$ and $\Ga$ is $G$-neighbour-transitive, then $\Ga$ is 
$G$-strongly incidence-transitive.
\item[(c)] If $G$ is primitive on $\V$ and $\Ga$ is $G$-strongly incidence-transitive, 
then $G$ is $2$-transitive on $\V$.
\end{enumerate}
\end{theorem}

In particular, if  $\Gamma$ is $G$-neighbour-transitive with $\delta(\Gamma)\geq3$ 
and $G$ is primitive on $\V$  then, by Theorem~\ref{flagtra}, $G$ 
is $2$-transitive on $\V$ and $\Gamma$ is strongly $G$-incidence-transitive.  
Application of the classification of the finite $2$-transitive
permutation groups opens up the possibility of classifying such codes. Moreover,
Theorem~\ref{flagtra} suggests that the possibly larger class of 
$G$-strongly-incidence-transitive codes (with $\delta(\Gamma)\geq2$) may also be 
analysed in this way.

To make progress with this analysis, 
we divide the finite $2$-transitive permutation groups according to whether or not 
they lie in an infinite family of $2$-transitive groups. Those which do not lie 
in an infinite family we call {sporadic}, and these cases are dealt with in \cite{NP},
yielding 27 strongly-incidence-transitive (code, group) pairs \cite[Table 1]{NP}.
In the rest of this paper we focus on the infinite families of finite 
$2$-transitive groups $G$.
As mentioned above, we do not treat the $2$-transitive actions of symplectic 
groups on quadratic forms, and indeed the open Problem~\ref{prob1} may be broadened 
to include the strongly-incidence-transitve case.

\begin{problem}\label{prob2}{ 
   Classify the $G$-strongly incidence-transitive codes $\Ga\subset J(v,k)$, 
where $G=\Sp(2n,2)$ and $v=2^{2n-1}\pm 2^{n-1}$.
}
\end{problem}
 
The other infinite families of $2$-transitive groups may be subdivided coarsely as 
in Table~\ref{Table1}. The `rank 1 case' is completely analysed in Section~\ref{sec-rank1}, and
we prove there the following classification result.

\begin{table}
\begin{center}
\begin{tabular}{lcl}\hline
\textbf{rank 1}&& the Suzuki, Ree and Unitary groups\\ 
\textbf{affine}&& $G\leq {\rm A}\GaL(\V)$ acting on $\V=\F_q^n$\\
\textbf{linear}&& $\PSL(n,q)\leq G\leq\PGaL(n,q)$ on $\PG(n-1,q)$\\ \hline
\end{tabular}
\end{center}
\caption{Other types of $2$-transitive permutation groups}\label{Table1} 
\end{table}

\begin{theorem}\label{thm-r1}
Suppose that $\Ga\subset \binom{\V}{k}$ is $G$-strongly 
incidence-transitive, where $G$ is $2$-transitive of rank $1$ type on $\V$.
Theneither
\begin{enumerate}
 \item[(a)]   $\PSU(3,q)\leq G\leq \PGaU(3,q)$, and either $\Ga$ or $\ov{\Ga}$ 
is the classical unital with $\delta(\Ga)=q$, as in Example~$\ref{ex-u}$; or
\item[(b)] $G=\PSU(3,3).2$, $k=12$ or $16$, and 
$\Ga$ or $\ov{\Ga}$ is the set of `bases' with $\delta(\Ga)=6$.
\end{enumerate}
\end{theorem}

The affine and linear cases are analysed in Sections~\ref{sec-aff} and~\ref{sec-linear}.
Propositions~\ref{aff1},~\ref{aff2},~\ref{lem-2dim}, and~\ref{lem-lin2} of these 
sections yield the following information about the possible strongly 
incidence-transitive codes $\Ga$ in these cases. Here a codeword 
$\gamma\in\Gamma$ is a subset of points of an affine or projective space 
$\V$. We say that $\gamma$ is of 
\emph{class} $[m_1,m_2,m_3]_1$ if each (affine or projective) line meets $\gamma$
in $m_1, m_2$ or $m_3$ points. 

\begin{theorem}\label{thm-al}
Suppose that $\Ga\subset \binom{\V}{k}$ is $G$-strongly 
incidence-transitive, where $G$ is $2$-transitive of affine or linear type on $\V$.
Let $\ga\in\Ga$. Then either $G$ and $\ga$ or $\ov\ga$ are as in one of the lines of Table~
$\ref{tbl-2trans}$, or one of the following holds.
\begin{enumerate}
\item[(a)] $G=\AGL(n,4)$ and $\V=\F_4^n$ with $n\geq2$, $\ga$ is of class $[0,2,4]_1$, 
and $\frac{4^n+2}{3}\leq k\leq \frac{2(4^n-1)}{3}$; or

\item[(b)] $G=\AGL(n,16)$ and $\V=\F_{16}^n$ with $n\geq2$, and replacing $\ga$ by 
$\ov\ga$ if necessary, $\ga$ is of class $[0,4,16]_1$, 
and $\frac{16^n+4}{5}\leq k\leq \frac{4(16^n-1)}{15}$; or

\item[(c)] $G=\PGaL(n,q)$, $\V=\PG(n-1,q)$ of size $v=\frac{q^n-1}{q-1}$
with $n\geq3$ and, replacing $\ga$ by 
$\ov\ga$ if necessary, $\ga$ is of class $[0,x, q+1]_1$, where either
\begin{enumerate}
 \item[(i)] $x=2$ and $\frac{v-1}{q}+1\leq k\leq \frac{2(v-1)}{q}$, or
\item[(ii)] $x=q_0+1$ where $q=q_0^2$, and $\frac{v-1}{q_0}+1\leq k\leq \frac{v-1}{q_0} + \frac{v-1}{q}$. 
\end{enumerate}
\end{enumerate} 
\end{theorem}

\begin{center}
\begin{table}
% use packages: array
\footnotesize
\begin{tabular}{ccclcl}\hline
$v$  &$\min\{k,v-k\}$    &$\delta(\Ga)$& \ \qquad $\ga$ or $\ov\ga$ & $G$ & Reference    \\ \hline \hline
     &       &             & Affine&     &              \\  \hline
$q^n$& $q^s$ &$q^s-q^{s-1}$&$s$-subspace & ${\rm A}\GaL(n,q)$&Example~\ref{ex-aff2}    \\ 
$16$& $4$ &$3$&Baer-subline & ${\rm A}\GaL(1,16)$&Proposition~\ref{aff1}    \\  \hline
%$4^n$& $2.4^{n-1}$ &???    &$\{v\in\V|v_1\in\F_2\}$ & ${\rm A}\GaL(n,4)$&Example~\ref{ex-aff3}    \\ 
%$16^n$& $4.16^{n-1}$ &???    &$\{v\in\V|v_1\in\F_4\}$ & ${\rm A}\GaL(n,16)$&Example~\ref{ex-aff3}    \\ 
%$16$  & $6$   &     ???   &hyperoval & ${\rm A}\GaL(2,4)$&Example~\ref{ex-aff4}    \\ \hline
     &       &             & Linear&     &              \\  \hline 
$\frac{q^n-1}{q-1}$& $\frac{q^s-1}{q-1}$ &$q^{s-1}$&$(s-1)$-subspace & ${\rm P}\GaL(n,q)$&Example~\ref{ex-lin1}\\
$q_0^2+1$& $q_0+1$ &$q_0$&Baer-subline & ${\rm P}\GaL(2,q_0^2)$&Example~\ref{ex-lin-2dim}    \\  \hline
%      &       &             & Unitary&     &              \\  \hline  
% $q^3+1$& $q+1$ &$q-1$&block of unital & ${\rm P\Ga U}(3,q)$&Example~\ref{ex-u}    \\ 
% $28$& $12$ &$6$&base & ${\rm PGU}(3,3)$&Proposition~\ref{lem-lin2}    \\  \hline 
 &&&&& \\
 \end{tabular}
\normalsize
\caption{Strongly incidence-transitive codes for Theorem~\ref{thm-al}.}\label{tbl-2trans}
 \end{table}
 \end{center}

There are more examples of strongly
incidence-transitive codes with affine or linear groups than the ones
listed in Table~\ref{tbl-2trans}. Example~\ref{ex-aff4} gives another
such code that satisfies Theorem~\ref{thm-al}~(a)
with $n=2$; in that example $\ga$ is the famous 6-point 
2-transitive hyperoval. I asked about the possible  
structures of  affine and projective point sets $\ga$ of classes $[0,x,q]_1$ 
or $[0,x,q+1]_1$ during a lecture I gave in 2012 in Ferrara 
at a Conference on Finite Geometry in honour of Frank De Clerck.
Nicola Durante, who was present, harnessed the known results about such 
subsets and developed them a great deal further in~\cite{Dur}.
He was able to  classifiy all such sets with the relevant symmetry properties, 
and hence classify all strongly incidence-transitive codes
in cases (a)--(c) of Theorem~\ref{thm-al}. We summarise his findings in Remark~\ref{rem-Dur}.

\begin{remark}\label{rem-Dur}
{\rm 
In the affine case, Durante classified geometrically all point subsets of 
$\AG(n,q)$ of class $[0,\sqrt{q},q]_1$ in Propositions 2.3, 3.6, 
Corollary 2.4 and Theorems 3.13 and 3.15 of \cite{Dur}. He used this information
to classify in \cite[Theorem 3.18]{Dur} all such subsets $\ga$ 
satisfying the coniditions in Theorem~\ref{thm-al}~(a) and~(b).  
For $q=4$, that is, for Theorem~\ref{thm-al}~(a), the examples for $\ga$ or $\ov\ga$
are (i) cyclinders with base the $2$-transitive hyperoval or its 
complement in $\AG(2,4)$, and (ii) unions of two parallel hyperplanes.
For $q=16$, that is, for Theorem~\ref{thm-al}~(b), the examples for 
$\ga$ or $\ov\ga$ are unions of four parallel hyperplanes of $\AG(n,16)$
(so that four-point intersections of affine lines with the set are Baer sub-lines).
In the projective case, that is, the case of Theorem~\ref{thm-al}~(c), 
Durante~\cite[Theorems 3.2 and 3.3]{Dur} drew together known 
results about sets of projective points of class $[0,x,q+1]_1$, 
and proved  that there are no
examples satisfying the conditions in Theorem~\ref{thm-al}~(c).  
} 
\end{remark}

We conclude this introductory section with a short commentary on the 
examples in Theorem~\ref{notprim}, and the conditions in Theorem~\ref{flagtra}.

\subsection{Remarks on the examples in Theorem~\ref{notprim}}\label{rem1}

(a)\quad 
The codes in Example~\ref{ex-intrans} are precisely the 
completely regular codes of `strength zero' classified 
by Meyerowitz~\cite{Meyerowitz1,Meyerowitz2}. On the 
other hand, some, but not all, of the codes in 
Examples~\ref{ex:imprim} and {\rm\ref{ex:imprim2}} 
are completely regular.  Further, some, but not all, 
of the codes in Example~\ref{ex:imprim2} are examples 
of groupwise complete designs constructed by 
Martin~\cite{Martin94}; and some of the codes 
in Example~\ref{rem:imprim} were discovered as 
completely transitive designs in \cite{GP}. 
See  Remarks~\ref{rem:imprim} and~\ref{rem:groupwise} 
for more details.

\medskip\noindent
(b)\quad Most of the neighbour-transitive codes
classified in Theorem~\ref{notprim} have minimum distance $\de(\Ga)=1$, the exceptions being 
the codes consisting of a single codeword, the blocks of a partition, or 
`blow-ups' of smaller neighbour-transitive codes. See 
Lemmas~\ref{lem-intrans}, \ref{lem-imprim} and \ref{lem:imprimex2}.

\subsection{Remarks on the transitivity properties  
in Theorem~\ref{flagtra}}\label{rem-flag-tra}

(a)\quad As we mentioned above, it is  possible for  a code to be 
incidence-transitive but not strongly incidence-transitive, so the condition on 
$\delta(\Ga)$ in part (a) of Theorem~\ref{flagtra} cannot be dropped.

\medskip\noindent
(b)\quad 
Strongly incidence-transitive codes exist with minimum distance 
as small as $2$ (see Example~\ref{ex:O43}), so the converse of part (b) 
of Theorem~\ref{flagtra} is false.

\medskip\noindent
(c)\quad There exist $G$-imprimitive, $G$-strongly incidence transitive codes (see 
Lemma~\ref{lem:imprimex2}), so the primitivity condition on $G$ cannot be dropped 
from part (c) of Theorem~\ref{flagtra}.
In addition there exist neighbour-transitives codes which are 
neither incidence-transitive nor strongly incidence transitive, and for which the automorphism 
group is 2-transitive on $\V$ (see Example~\ref{ex:de1b}). Thus the converse of part (c) is false.

\medskip\noindent
(d)\quad 
Each code $\Ga\subset\binom{\V}{k}$ has a kind of dual  defined  as follows. 
For a subset $\alpha \subset\V$, we write $\overline\alpha=\V\setminus\alpha$. 
The \emph{complementary code} of $\Ga$ is $\overline\Ga:=\{\overline\ga\,|\,\ga\in\Ga\}$. It is a code in 
$J(v,v-k)$ with neighbour set $\{\overline\ga\,|\,\ga\in\Ga_1\}$. Moreover $\de(\overline\Ga)=\de(\Ga)$, 
and any of the properties of neighbour-transitivity, incidence-transitivity, or strong 
incidence-transitivity holds for $\Ga$ if and only if it holds for $\overline\Ga$. 
Thus we may assume that $k\leq v/2$ for the proof of Theorem~\ref{flagtra}, and we do this 
also at various other stages of our investigation.

%%$\mathbb{\vp}, \mathbf{\vp}, \vp,\quad $
%{\boldmath{$\vp$}}, $\ov\vp, \underline{\vp}, 
% \underline{ \underline{\vp}}$

%$\mathbf{X}\quad \mathbb{X},\quad, X$

\section{Preliminaries and examples}\label{sect:prelims}

\subsection{Completely regular and completely transitive codes in graphs}\label{creg}
A code $\Gamma$ in a connected graph $\Sigma$ determines a 
distance partition $\{\Ga_0,\dots,\Ga_{r-1}\}$ of the vertex 
set of $\Sigma$, where $\Ga_0=\Ga$ and, for $i>0$, $\Ga_i$ is 
the set of vertices  which are at distance $i$ from at least 
one codeword in $\Ga$, and at distance at least $i$ from each 
codeword. For the last non-empty set $\Ga_{r-1}$, the parameter 
$r$ is called the covering index of $\Ga$. The code is 
\emph{completely regular} if the partition $\{\Ga_0,\dots,
\Ga_{r-1}\}$ is equitable, that is to say, for any $i,j\in
\{0,\dots,r-1\}$ and any $\gamma\in \Ga_i$, the number of 
vertices of $\Ga_j$ adjacent to $\gamma$ is independent of 
the choice of $\gamma$ in $\Ga_i$, and depends only on $i$ 
and $j$.  Further, $\Ga$ is called \emph{completely transitive} 
if the setwise stabiliser of $\Ga$ in $\Aut(\Sigma)$ (which 
automatically fixes each of the $\Ga_i$ setwise) is transitive 
on each $\Ga_i$.

\subsection{Codes in Johnson graphs: notation and small examples}\label{notn}

It is useful to denote the arc set of the Johnson graph $J(v,k)$ by $J$, that is, 
\begin{equation}\label{eq-j}
 J:=\{(\al,\be)\,|\, \al,\be\in\binom{\V}{k},\,|\al\cap\be|=k-1\}
\end{equation}
and for $\al\in\binom{\V}{k}$, to write $J(\al)=\{\be\,|\,(\al,\be)\in J\}$.
For a code $\Ga\subset \binom{\V}{k}$ with neighbour set $\Ga_1$, 
pairs $(\ga,\ga_1)\in\Ga\times\Ga_1$ with $d(\ga,\ga_1)=1$ in $J(v,k)$ 
are called the \emph{incidences} of $\Gamma$: the set of incidences is the 
subset $J\cap (\Ga\times\Ga_1)$ and, for $G\leq\Aut(\Ga)$,  $\Ga$ is $G$-incidence-transitive 
if $G$ is transitive on $J\cap (\Ga\times\Ga_1)$.

For a subset $\al\subseteq\V$ we write $\ov\al:=\V\setminus\al$, so that,
for $G\leq\Aut(\Ga)\cap\Sym(\V)$,  
$\Ga$ is $G$-strongly incidence-transitive if $G$ is transitive on $\Ga$ 
and, for $\ga\in\Ga$, $G_\ga$ acts 
transitively on $\ga\times\ov\ga=\{(\u,\w)\,|\,\u\in\ga, \w\in\ov\ga\}$. 
%If this holds for $G=\Aut(\Ga)$ we say that $\Ga$ is strongly incidence-transitive.
First we verify that this concept is indeed a strengthening of incidence-transitivity.

\begin{lemma}\label{str-flagtra}
If $\Ga$ is $G$-strongly incidence-transitive, %and $\Ga\ne\binom{\V}{k}$, 
then $\Ga$ is $G$-incid\-ence-transitive and $\de(\Ga)\geq2$. 
%and $G$ acts transitively on  $J\cap (\Ga\times\Ga_1)$. 
%Moreover, there exist incid\-ence-transitive codes $\Ga$ that are not strongly incidence-transitive.
\end{lemma}

\begin{proof}
Since  $\Ga\ne\binom{\V}{k}$ and the graph $J(v,k)$ is connected, the neighbour set $\Ga_1$
is non-empty. Then since $G$ is transitive on $\Ga$, it follows that, for each $\ga\in\Ga$, 
$d(\ga,\ga_1)=1$ for some $\ga_1\in\Ga_1$. Moreover, since $G_\ga$ is transitive on 
$\ga\times\ov\ga$ it follows that $G_\ga$ is transitive on $J(\ga)$, and hence $J(\ga)\subseteq\Ga_1$.
Thus  $\de(\Ga)\geq2$, $\Ga_1=\cup_{\ga\in\Ga}J(\ga)$, and $G$ is transitive on  $J\cap (\Ga\times\Ga_1)$,
the set of incidences. 
\end{proof}

Next we give several examples that illustrate various differences between the transitivity concepts.

\begin{example}\label{ex:de1b}{\rm 
Let $\V$ be the set of points of the projective line ${\rm PG}_1(q)$, where
$q\equiv 1\pmod{4}$ and $q>5$, and let $G={\rm PSL}(2,q)$. Let $\Ga$ be 
one of the two $G$-orbits on 3-subsets of $\V$. Then $\Ga_1$ is the 
other $G$-orbit on $3$-subsets. Thus $\Ga$ is neighbour-transitive, but not
incidence-transitive,  and $\de(\Ga)=1$. 
}
\end{example}

\begin{example}\label{ex1}{\rm
Let $|\V|=9, k=3$, and let  $\U=\{U_1|U_2|U_3\}$ be a partition of $\V$ with $3$ 
parts of size $3$. Let $\Ga=\Ga'\cup\U$, where $\Ga'$ is the set of 
all $3$-subsets that contain  
one point from each part of $\U$.
Let $\Delta$ be the set of $3$-subsets containing two points from one part of $\U$ and one point 
from a second part of $\U$. Then $\Aut(\Ga)=\Aut(\Delta)$ is the stabiliser $S_3\wr S_3$ 
of $\U$ in $\Sym(\V)$, and $\U, \Ga', \Delta$ are its three orbits in $\binom{\V}{3}$. 
The two codes $\Ga, \Delta$ in $J(9,3)$ have neighbour sets $\Delta$ and $\Ga$ respectively. 
Thus $\Delta$ is  code-transitive but not neighbour-transitive, while $\Aut(\Ga)$ 
is transitive on $\Ga_1=\Delta$ but not on $\Ga$. These codes have $\de(\Delta)=\de(\Ga)=1$.
}
\end{example}

The following is an incidence-transitive example which is not strongly incidence-transitive, for which the automorphism group has a natural proper
subgroup with weaker transitivity properties than those of the full group, but is still transitive on the neighbours. 

\begin{example}\label{ex:de1}{\rm 
Let $\V$ be the point set  of the projective plane ${\rm PG}_2(q)$, where
$q>5$, let $\lambda$ be a line, and let $G$ be the setwise stabiliser of $\lambda$ in
${\rm PGL}(3,q)$. Let $\Ga$ be the set of all $4$-subsets of $\lambda$. Then $\Ga_1$
consists of all $4$-subsets $\ga_1$ such that $|\ga_1\cap\lambda|=3$, the distance 
$\de(\Ga)=1$, and $\Aut(\Ga)=S_{q^2}\times S_{q+1}$ is incidence-transitive, but not 
strongly incidence-transitive. On the other hand $G$ is transitive on 
$\Ga_1$, while $\Ga$ is a union of at least two $G$-orbits. 
}
\end{example}

Finally we give an example from orthogonal geometry of a strongly 
incidence-transitive code with minimum distance 2, showing that such codes 
with minimum distance 2 do exist. An ovoid in a projective space $\PG(3,q)$ is
a subset of $q^2+1$ points that meets each line in at most two points.

\begin{example}\label{ex:O43}{\rm
Let $G={\rm GO}^-_4(3)$ acting on an ovoid $\V\subset {\rm PG}_3(3)$. 
Then $|\V|=10$. Let
$\Ga$  be the set of 4-subsets of $\V$ (called `circles') 
obtained as intersections of $\V$ with secant planes 
%of ${\rm PG}_3(3)$ 
lying on external lines. Then $|\Ga|=30$, 
$\Ga$ is a $G$-orbit, and $G$ is transitive on $\binom{\V}{4}\setminus \Ga$. 
Thus $\Ga_1=\binom{\V}{4}\setminus \Ga$. Now the subsets in $\Ga$ form the 
block set of a 
$3-(10,4,1)$ design implying that $\de(\Ga)\geq2$, and in fact $\de(\Ga)=2$. 
For a circle $\ga\in\Ga$, $G_\ga={\rm PGL}_2(3)\cong S_4$ is transitive on $\ga$, and
for $\u\in\ga$, $G_{\ga,\u}$ is transitive on the 6 points of $\V\setminus\ga$.
Thus $\Ga$ is $G$-strongly incidence-transitive.
}
\end{example}

\section{The intransitive neighbour-transitive codes}\label{sect:intrans}

Several natural families of neighbour-transitive designs have an 
automorphism group in $\Sym(\V)$ that is intransitive on the underlying point set $\V$. 
We describe these families and prove that 
they are the only exampes with automorphism groups intransitve on $\V$.

\begin{example}\label{ex-intrans}{\rm
Let $\U$ be a proper, non-empty subset of $\V$ and $2\leq k\leq v-2$, where $v=|\V|$. 
Define $\Ga\subset\Xv$ as follows.
\begin{enumerate}
 \item[(a)] If $|\U|>k$, let $\Ga$ be the set of all $k$-subsets of $\U$.
 \item[(a)] If $|\U|=k$, let $\Ga$ be the singleton set $\{\U\}$.
 \item[(c)] If $|\U|<k$, let $\Ga$ be the set of all $k$-subsets of $\V$ containing $\U$.
\end{enumerate}
}
\end{example}

As noted in Subsection~\ref{rem1}, these are the completely regular codes of 
`strength zero' classified by Meyerowitz~\cite{Meyerowitz1,Meyerowitz2}.
We examine neighbour\-transitive subgroups of automorphisms of these designs. 
For $\U\subseteq\V$, we write $\ov\U:=\V\setminus\U$, and 
we denote the setwise stabiliser of $\U$ in $\Sym(\V)$ by $\Stab(\U)$.
For a positive integer $k$, a group of permutations of a set $\U$ is said 
to be \emph{$k$-homogeneous} on $\U$
if it is transitive on the $k$-subsets of $\U$.

\begin{lemma}\label{lem-intrans}
Let $\U, k, \Ga$ be as in Example~{\rm\ref{ex-intrans}}. Then $\Ga$ 
is neighbour-transitive, $\Aut(\Ga)\cap\Sym(\V)=\Stab(\U)$ is intransitive on $\V$,
%and $\Aut(\Ga)\leq\Sym(\V)$ unless we are in case (b) with $k=v/2$.  
Also, if $|\U|\ne k$ then  $\de(\Ga)=1$.  Moreover,
if $G\leq\Aut(\Ga)\cap\Sym(\V)$, then $\Ga$ is $G$-neighbour-transitive 
if and only if  $\Ga$ is $G$-incidence-transitive, 
if and only if  
\begin{enumerate}
 \item[(a)] $|\U|>k$, and $G$ is transitive on both $\binom{\U}{k}$ and $\binom{\U}{k-1}\times\ov\U$; or
 \item[(b)] $|\U|=k$, and $G$ is transitive on $\U\times\ov\U$; or
\item[(c)] $|\U|<k$, and $G$ is transitive on both $\binom{\ov\U}{k-|\U|}$ and  
$\U\times \binom{\ov\U}{k-|\U|+1}$.
%transitive on $\U\times(\ov\U\setminus\ga)$. $G$ is $(k-|\U|)$-homogeneous on $\ov\U$ and
\end{enumerate}
Further, $\Ga$ is $G$-strongly incidence transitive if and only if $G, k$ are as in part (b). 
\end{lemma}

We remark that, in case (c) when $k\leq (v+1)/2$, transitivity on
$\binom{\ov\U}{k-|\U|}$ follows from transitivity on $\U\times \binom{\ov\U}{k-|\U|+1}$
(see  \cite[Theorem 9.4A(ii)]{DM}).

\begin{proof} It follows from the definition of $\Ga$ that $\Aut(\Ga)\cap\Sym(\V)=\Stab(\U)$ with 
orbits $\U$ and $\ov\U$ in $\V$, and also, if $k\ne v/2$ 
%then $\Aut(\Ga)\leq\Sym(\V)$, while if $k=v/2$ $|\U|\ne k$, 
that  $\de(\Ga)=1$.  Note 
that the neighbour set  $\Ga_1$ consists of all $k$-subsets $\ga_1$ of $\V$ such that 
$|\ga_1\cap\U|=k-1$ if $|\U|\geq k$, or  $|\ga_1\cap\U|=|\U|-1$ if $|\U|< k$.  Let 
$G\leq\Aut(\Ga)\cap\Sym(\V)$ and $\ga_1\in\Ga_1$. If  $|\U|= k$, then the criterion for 
neighbour-transitivity in (b) is clearly necessary and sufficient, and moreover it is 
equivalent to $G$-strong incidence transitivity.

Suppose next that  $|\U|> k$. Then $G$ is transitive on $\Ga$ if and only if  
$G$ is $k$-homogeneous on $\U$. Here $|\ga_1\cap\U|=k-1$, so $G$ is transitive 
on $\Ga_1$ if and only if $G$ is  $(k-1)$-homogeneous on $\U$ and $G_{\ga_1\cap\U}$ 
is transitive on $\ov\U$. Thus the conditions given in (a) are necessary and sufficient 
for $G$ to be neighbour-transitive, and to be incidence-transitive. Here  $\delta(\Gamma)=1$, 
and hence by Lemma~\ref{str-flagtra}, $\Ga$ is not strongly 
incidence-transitive.

Suppose finally that  $|\U|< k$. Then $G$ is transitive on $\Ga$ if and only if  
$G$ is  $(k-|\U|)$-homogeneous on $\ov\U$. Here $|\ga_1\cap\U|=|\U|-1$, so $G$ 
is transitive on $\Ga_1$ if and only if $G$ is transitive on $\U$ and $G_{\ga_1\cap\U}$ 
(which is the stabiliser of the unique point of $\U$ not in $\ga_1$) is  $(k-|\U|+1)$-homogeneous 
on $\ov\U$; this is equivalent to  $G$ being transitive on  $\U\times \binom{\ov\U}{k-|\U|+1}$. 
% Moreover, since $k\leq\frac{v}{2}$, it follows from \cite[Theorem 9.4A(ii)]{DM} that the condition 
% that $G_{\ga_1\cap\U}$ is  $(k-|\U|+1)$-homogeneous on $\ov\U$ implies that  $G_{\ga_1\cap\U}$ is  
% $(k-|\U|)$-homogeneous on $\ov\U$, and hence that $G$ is $(k-|\U|)$-homogeneous on $\ov\U$. 
Thus the condition in (c) is necessary and sufficient for neighbour transitivity, and for 
incidence-transitivity. Again  $\delta(\Gamma)=1$, 
and hence by Lemma~\ref{str-flagtra}, $\Ga$ is not strongly 
incidence-transitive.
\end{proof}

Now we classify the codes $\Ga$ admitting a neighbour-transitive, intransitive group. 

\begin{proposition}\label{intrans}
Suppose that $\Ga\subset \Xv$, where $2\leq k\leq |\V|-2$, and $\Ga$ admits a 
neighbour-transitive subgroup of $\Sym(\V)$ that is intransitive on $\V$. Then
$\Ga$ is as in Example~{\rm\ref{ex-intrans}}.
\end{proposition}

\begin{proof}
  Let $G\leq\Aut(\Ga)\cap\Sym(\V)$ be neighbour-transitive on $\Ga$ and intransitive on $\V$. 
Suppose first that some codeword $\ga\in\Ga$ contains a $G$-orbit, and let  $\U$ be the largest 
$G$-invariant subset of $\ga$. Since $G$ is transitive on $\Ga$ it follows that $\U$ is 
contained in each codeword of $\Ga$. Now there exists some neighbour $\ga_1\in\Ga_1$ that 
does not contain $\U$. If $\Ga$ did not contain every $k$-subset containing $\U$, then 
there would also be a neighbour $\ga_1'\in\Ga_1$  containing $\U$. However, since $G$ 
fixes $\U$ setwise, no element of $G$ can map $\ga_1$ to $\ga_1'$, which is a contradiction. 
Hence $\Ga$ consists of all $k$-subsets that contain $\U$, as in Example~\ref{ex-intrans}.

Thus we may assume that a codeword $\ga\in\Ga$ does not contain any $G$-orbit in $\V$. 
Let $\U$ be a $G$-orbit that meets $\ga$, and let $\u\in\ga\cap\U$ and $\u'\in\U\setminus
\ga$. Set $k':=|\ga\cap\U|$. Since $G$ is transitive on $\Ga$ it follows that every codeword 
meets $\U$ in $k'$ points. Since $\ga$ contains no $G$-orbit it follows that there exists a 
point $\v\in\ov\U\setminus\ga$. The $k$-subset $\ga_1:=(\ga\setminus\{\u\})\cup\{\v\}$ meets 
$\U$ in $k'-1$ points and hence does not lie in $\Ga$. Since $d(\ga,\ga_1)=1$ it follows that 
$\ga_1\in\Ga_1$, and since $G$ is transitive on $\Ga_1$, all neighbours must meet $\U$ in 
$k'-1$ points. Suppose that $k'<k$ so that there exists $\v'\in\ga\setminus\U$. Then the 
$k$-subset $\ga':=(\ga\setminus\{\v'\})\cup\{\u'\}$ meets $\U$ in $k'+1$ points and hence 
does not lie in either $\Ga$ or $\Ga_1$, contradicting the fact that $d(\ga,\ga')=1$. 
Hence $k'=k$, that is to say, $\ga\subset\U$. Thus each codeword is contained in $\U$ 
and each neighbour meets $\U$ in $k-1$ points. Now each $k$-subset $\gamma'\subset\U$ 
is connected to $\gamma$ by a path in the Johnson graph $J(|\U|,k)$ based on $\U$, with 
each vertex on the path a $k$-subset of $\U$. It follows that each vertex on the path 
is a codeword, and in particular $\gamma'\in\Ga$. Hence $\Ga$ consists of all $k$-subsets 
of $\U$, as in Example~\ref{ex-intrans}.
\end{proof}

\section{The imprimitive neighbour-transitive codes}\label{sect:imprim}

Additional natural families of neighbour-transitive codes are based 
on a partition $\U=\{U_1|U_2|\dots|U_b\}$ of the underlying set $\V$, with $b$ 
equal-sized parts $U_i$ of size $a$, where $v=ab, a>1, b>1$. 
We introduce the notion of the \emph{$\U$-type}  $\t_\U(\ga)$ of a subset $\ga$ of $\V$ to
describe how $\ga$ intersects the various parts of $\U$, namely $\t_\U(\ga)$ is 
the multiset $\{1^{m_1},2^{m_2},\dots,a^{m_a}\}$ of size $b=\sum_{i=1}^a m_i$, where 
$\sum_{i=1}^aim_i=|\ga|$, each $m_i\geq0$ and exactly $m_i$ of the 
intersections $\ga\cap U_1,\dots,\ga\cap U_b$ 
have size $i$. If some $m_i=0$, we usually omit the entry $i^{m_i}$ from $\t_\U(\ga)$. 
For example, if $\ga\subset U_b$ and $|\ga|=k$, we write $\t_\U(\ga)$ as $\{k\}$. 

We give two constructions in Subsection~\ref{sub-imprimex} for 
codes admitting a group of automorphisms 
that is both neighbour-transitive on the code, and also 
transitive and imprimitive on $\V$. 
We prove in Subsection~\ref{sub-imprimclass} 
that all such codes arise in one of these ways.

\subsection{Constructions of imprimitive codes}\label{sub-imprimex}
Several families of codes are defined as the sets of all $k$-subsets
of certain $\U$-types.

\begin{example}\label{ex:imprim}{\rm
Let  $\U=\{U_1|U_2|\dots|U_b\}$ be a partition of the $v$-set $\V$ with $b$ 
parts of size $a$, where $v=ab, a>1, b>1$, and let $2\leq k\leq v-2$.
For $\t$ as in one of the lines of Table~\ref{tbl-utype},  
let $\Ga(a,b;\t)$ consist of all $k$-subsets $\ga$ of $\V$ with $\t_\U(\ga)=\t$.
} 
\end{example}

 \begin{center}
\begin{table}
% use packages: array
\begin{tabular}{clll}\hline
Line &$\t$  &Conditions    &$\mathcal{U}$-Type for $\Ga_1$\\ \hline 
1 &$\{k\}$ &$k\leq a$   &$\{1,k-1\}$ \\ 
2 &$\{k-v+a,a^{b-1}\}$&$k\geq v-a$ &$\{k-v+a+1,a-1,a^{b-2}\}$ \\
3 &$\{1^k\}$ &$k\leq b$ &$\{1^{k-2},2\}$  \\
4 &$\{(a-1)^{v-k},a^{b-v+k}\}$& $k\geq v-b$& $\{a-2,(a-1)^{v-k-2},a^{b-v+k+1}\}$\\ 
5 &$\{c^b\}$  &$k=cb$, with  & $\{c-1,c^{b-2},c+1\}$   \\
  &           & $1<c\leq a-1$&                         \\  
6 &$\{\frac{k-1}{2},\frac{k+1}{2}\}$  &$k$ odd, with &$\{\frac{k-3}{2},\frac{k+3}{2}\}$ \\ 
  &                                   & $a\geq3$, $b=2$  &    \\
7 &$\{1,2^{(k-1)/2}\}$  &$k$ odd, with & $\{1^3,2^{(k-3)/2}\}$     \\ 
  &                     & $a=2$, $b\geq3$ & \\  \hline
&&
 \end{tabular}
\caption{$\mathcal{U}$-Types for $\Ga=\Ga(a,b;\t)$ and $\Ga_1$ in Example~\ref{ex:imprim}.}\label{tbl-utype}
 \end{table}
 \end{center}

\begin{remark}\label{rem:imprim}{\rm
Some, but not all, of the neighbour-transitive codes in Example~\ref{ex:imprim} 
are completely regular. A discussion is given in \cite[Section 2]{Martin94}.  
Moreover the codes in Line 5 of Table~\ref{tbl-utype} are completely transitive, 
not just neighbour-transitive. Also those in Line 1 of Table~\ref{tbl-utype} are 
completely transitive if either $b=2$ or $k=3$. These were  discovered as completely 
transitive codes in \cite{GP} (see also \cite[page 181]{Martin94}). 
} 
\end{remark}

We note that, because of the conditions given in Table~\ref{tbl-utype}, 
no code arises from more than one line. We denote the stabiliser 
in $\Sym(\V)$ of the partition $\U$ by $\Stab(\U)\cong S_a\wr S_b$. 
It is transitive and imprimitive on $\V$, since $a>1, b>1$.
We prove that each of these codes $\Ga$ is incidence-transitive and hence, in particular, is neighbour-transitive.

\begin{lemma}\label{lem-imprim}
Let $\U, k, \t, \Ga=\Ga(a,b;\t)$, be as in Example~{\rm\ref{ex:imprim}}. Then $\Ga$ 
is $G$-incidence-transitive, where $G=\Aut(\Ga)\cap\Sym(\V)=\Stab(\U)$ is transitive 
and imprimitive on $\V$. 
Also  $\de(\Ga)=1$, except for Line~$1$ of 
Table~{\rm\ref{tbl-utype}} with $k=a$, and in this exceptional case $\de(\Ga)=k$.  
\end{lemma}

\begin{proof} 
Set $S=\Stab(\U)$, let $\t$ be as in one of the lines of Table~\ref{tbl-utype},
and let $\Ga=\Ga(a,b;\t)$. Clearly $\de(\Ga)=1$ unless we are in Line 1 with $k=a$, 
and in this case $\Ga=\U$ and $\de(\Ga)=k$.
In all cases it is not difficult to prove that $S=\Aut(\Ga)\cap\Sym(\V)$, that is $S=G$,
and that $S$ is transitive on $\Ga$.
Also it is not difficult to verify that in each case $\Gamma_1$ consists of all
$k$-sets with $\mathcal{U}$-type as in Table~\ref{tbl-utype}. It then 
follows that $\Ga$ is $G$-incidence transitive.
% 
% 
% Finally, we specify the neighbour set $\Ga_1$ in each case, and it is then easy
% to verify that $G$ is incidence-transitive on $\Ga$.
% For Line 1 of Table~\ref{tbl-utype}, the neighbour set $\Ga_1$  
% consists of all $k$-subsets of $\mathcal{U}$-type $\{1,k-1\}$. 
% For Line 2, $\Ga_1$  is the set of all $k$-subsets of 
% $\mathcal{U}$-type $\{1^{k-2},2\}$.  For Line 3,  
% since $k\leq v/2$, we have $c<a$, and so $\Ga_1$ 
% is the set of all $k$-subsets of $\mathcal{U}$-type 
% $\{c-1,c^{b-2},c+1\}$.
% For Line 4, since $3\leq k\leq v/2=a$, we have $\frac{k+3}{2}\leq k\leq a$. 
% Thus $\Ga_1$ is the set of all $k$-subsets of $\mathcal{U}$-type 
% $\{\frac{k-3}{2},\frac{k+3}{2}\}$. For Line 5, since $3\leq k\leq v/2=b$, 
% we have $\frac{k+3}{2}\leq b$, and $\Ga_1$ is the set of all $k$-subsets of 
% $\mathcal{U}$-type 
% $\{1^3,2^{(k-3)/2}\}$
\end{proof}

The second set of examples involves a code 
$\Ga_0$  in a smaller Johnson graph $J(b,k_0)$ based on a partition $\mathcal{U}$ of $\V$. 
We include the case $k_0=1$ for a uniform description and note that in this case the code in 
Example~\ref{ex:imprim2} also occurs in Example~\ref{ex:imprim}, namely in Line 1 
of Table~\ref{tbl-utype} with $k=a$. This is the only overlap between the two families of codes.

\begin{example}\label{ex:imprim2}
Let  $\U=\{U_1|U_2|\dots|U_b\}$ be a partition of the $v$-set $\V$ with $b$ 
parts of size $a$, where $v=ab, a>1, b\geq4$, and let 
$k=ak_0$ where $1\leq k_0\leq b-1$.
For $\Ga_0\subseteq \binom{\U}{k_0}$, let $\Ga(a,\Ga_0)$ be the 
set of all $k$-subsets of $\V$ of the form $\cup_{U\in\ga_0}U$,
for some $\ga_0\in\Ga_0$.
\end{example}

\begin{remark}\label{rem:groupwise}{\rm
These codes were studied by Martin \cite{Martin94} in the special case where 
$\Ga_0=\binom{\U}{k_0}$, and he called them \emph{groupwise complete designs}. 
It follows from  Lemma~\ref{lem:imprimex2} that essentially all the strongly 
incidence-transitive codes $\Ga(a,\Ga_0)$ arising from Example~\ref{ex:imprim2} 
with $\de(\Ga_0)=1$ are groupwise complete designs.
Martin~\cite[Theorem 2.1]{Martin94} determined precisely which groupwise complete 
designs are completely regular codes.  He proved further in \cite[Theorem 3.1]{Martin94} 
that, if $\Ga$ is completely regular with minimum distance at least 2, and if 
$\Ga$ is a 1-design but not a 2-design\footnote{That is, each point of $\V$ 
lies in a constant number of codewords ($k$-subsets) in $\Ga$, but some point 
pairs  lie in different numbers of codewords.}, then $\Ga$ is a groupwise complete 
design. Thus it follows from Lemma~\ref{lem:imprimex2} that most of the 
neighbour-transitive codes in Example~\ref{ex:imprim2} are {\it not} completely regular.
}
\end{remark}

\begin{lemma}\label{lem:imprimex2}
For $\Ga=\Ga(a,\Ga_0)$ as in Example~{\rm\ref{ex:imprim2}},
$\Aut(\Ga)$ contains $S_a\wr A$, where $A=\Aut(\Ga_0)\cap\Sym(\U)$, and $\de(\Ga)=a \de(\Ga_0)$.

\begin{enumerate}
\item[(a)] Moreover,  if $\Ga_0$ is $A$-strongly incidence-transitive, 
then  $\Ga$ is $(S_a\wr A)$-strongly incidence-transitive, and 
either  $\Ga_0=\binom{\U}{k_0}$ or $\de(\Ga_0)\geq2$.
\item[(b)]  Conversely, if $S_a\wr A$ is neighbour-transitive on $\Ga$, then 
either $\Ga_0$ is $A$-strongly incidence-transitive, or $a=2$ and $\de(\Ga_0)=1$. 
\end{enumerate}
\end{lemma}

\begin{proof} It follows from the definition of $\Ga$ that 
$\de(\Ga)=a \de(\Ga_0)$. Let $G:=S_a\wr A$.
Then, by definition, $\Ga$ is a $G$-invariant subset of 
$\binom{\V}{k}$, so $G\leq\Aut(\Ga)$.
Next suppose that  $\Ga_0$ is $A$-strongly incidence-transitive. 
Then $G$ is transitive on $\Ga$. Let $\ga_0\in\Ga_0$ and $\ga=\cup_{U\in\ga_0}U$. 
Then $G_\ga=S_a\wr A_{\ga_0}$, and since  $A_{\ga_0}$
is transitive on $\ga_0\times(\U\setminus\ga_0)$, then also 
$G_\ga$ is transitive on  $\ga\times(\V\setminus\ga)$. 
Thus  $\Ga$ is $G$-strongly incidence-transitive.
Suppose that $\de(\Ga_0)=1$. Then we may choose $\ga_0$ such that some adjacent $k_0$-subset 
$(\ga_0\setminus\{U\})\cup\{W\}\in\Ga_0$. Since  $\Ga_0$ is $A$-strongly 
incidence-transitive it follows 
that each $k_0$-subset of $\U$ adjacent to $\gamma_0$ in $J(b,k_0)$ also lies in $\Ga_0$. Moreover, 
this property is independent of $\ga_0$ since $A$ is transitive on $\Ga_0$.
It follows that  $\Ga_0=\binom{\U}{k_0}$, and (a) is proved.

Conversely suppose that $G$ is neighbour-transitive on $\Ga$.
For $\ga_0,\ga_0'\in\Ga_0$, let 
$\ga=\cup_{U\in\ga_0}U$ and $\ga'=\cup_{U\in\ga_0'}U$. Then some 
element $h\in G$ maps $\ga$ to $\ga'$, and hence the element of $A$
induced by $h$ maps $\ga_0$ to $\ga_0'$. Thus $A$ is transitive on $\Ga_0$.
Now take $\ga_0=\ga_0'$, let  $U,U'\in\ga_0$ and $W,W'\in\U\setminus\ga_0$ (possibly $U=U'$, and/or
$W=W'$). Let $\u\in U, \u'\in  U',\w\in W, \w'\in W'$, and define 
$\ga_1=\ga_1(\u,\w)=(\ga\setminus\{\u\})\cup\{\w\}$ and 
$\ga_1'=\ga_1'(\u',\w')=(\ga\setminus\{\u'\})\cup\{\w'\}$.
Both are at distance 1 from $\ga$ and hence $\ga_1,\ga_1'\in\Ga_1$  
(recalling that $\delta(\Gamma)\geq a>1$).
Since $G$ is transitive on $\Ga_1$ there exists $g\in G$ such that
$\ga_1^g=\ga_1'$. This element $g$ maps  $\ga\setminus U$ to $\ga\setminus U'$ 
(the unions of $\mathcal{U}$-classes contained in $\ga_1, \ga_1'$ respectively), and maps $\ga\cup W$ to $\ga\cup W'$ (the unions of $\U$-classes meeting $\ga_1,\ga_1'$ respectively).
Hence $g$ maps $\{U,W\}$ to $\{U',W'\}$. 
Suppose first that $a\geq3$. Then $U, U'$ are the only $\U$-classes
such that $|\ga_1\cap U|= |\ga_1'\cap U'|=a-1$,  and hence
$g$ maps $U$ to $U'$, and $W$ to $W'$. Hence 
$g\in G_{\ga_0}$ and we deduce that $A_{\ga_0}$ is transitive on 
$\ga_0\times(\U\setminus\ga_0)$, so  $\Ga_0$ is $A$-strongly incidence-transitive.

Finally suppose that $a=2$ and $\de(\Ga_0)\geq2$. We showed that $g$ maps $\eta:=\ga_0\cup\{W\}$ to $\ga_0\cup\{W'\}$, $\nu:=\ga_0\setminus\{U\}$  to $\ga_0\setminus\{U'\}$, and $\{ U,W\}$ to $\{U',W'\}$.
Suppose that $(U,W)^g=(W',U')$. Then $\ga_0^g\subset \eta^g=\ga_0\cup\{W'\}$ and $U'=W^g\not\in\ga_0^g$, so that $\ga_0^g=(\ga_0\setminus\{U'\})\cup\{W'\}$. This implies that $d(\ga_0,\ga_0^g)=1$ which is a contradiction since $\ga_0^g\in\Ga_0$ and we assumed that $\de(\Ga_0)\geq2$. Thus $(U,W)^g=(U',W')$.

Temporarily take $W'=W$. Then $g$ fixes $\eta$ and hence also $\ga_0$, since $W^g=W'=W$. Since $U, U'$ can be chosen arbitrarily in $\ga_0$, it follows that $A_{\ga_0}$ is transitive on $\ga_0$. Now instead take $U'=U$. Then $g$ fixes $\nu$, and hence also $\ga_0$, since $U^g=U'=U$. Since $W, W'$ can be chosen arbitrarily in $\U\setminus\ga_0$ we conclude that $A_{\ga_0,U}$ is transitive on $\U\setminus\ga_0$. Hence $\Ga_0$ is $A$-strongly incidence-transitive.
\end{proof}

\subsection{Classification of the imprimitive codes}\label{sub-imprimclass}

Now we classify the codes $\Ga$ admitting an imprimitive neighbour-transitive subgroup of $\Sym(V)$. 

\begin{proposition}\label{imprim}
Suppose that $\Ga\subset \Xv$, where $2\leq k\leq |\V|-2$, and $\Ga$ admits a 
neighbour-transitive subgroup of $\Sym(\V)$ that is transitive and imprimitive on $\V$. Then
$\Ga$ is as in Example~{\rm\ref{ex:imprim}} or~{\rm\ref{ex:imprim2}}.
\end{proposition}

\begin{proof} Let $G\leq\Aut(\Ga)$ be neighbour-transitive and 
imprimitive on $\V$, and let $\U=\{U_1|U_2|\dots|U_b\}$ be a 
$G$-invariant partition of $\V$ with $b$ 
parts of size $a$, where $v=ab, a>1, b>1$.
Choose $\ga\in\Ga$, set $e_i:=|\ga\cap U_i|$ for each $i$, and re-label the 
$U_i$ so that $e_1\geq e_2\geq\dots\geq e_b$. Then $\ga$ is of 
$\mathcal{U}$-type $\t_\U(\ga)=\{e_1, e_2,\dots,e_b\}$. 
We examine various $k$-subsets $\be$ of $\V$ such that $d(\ga,\be)=1$. If
$\t_\U(\be)\ne \t_\U(\ga)$, then $\be\not\in\Ga$ since $G$ is transitive on $\Ga$ and 
preserves the $\mathcal{\U}$-types of $k$-subsets of $\V$. Hence $\be\in\Ga_1$. 
Moreover, since $G$ is transitive on $\Ga_1$, there do not exist two  
$k$-subsets $\be_1$ and $\be_2$ with  $d(\ga,\be_1)=d(\ga,\be_2)=1$ such that
$\t_\U(\be_1), \t_\U(\be_2)$ and $\t_\U(\ga)$ are pairwise distinct. 
We use the following notation. \emph{
For each $i$ such that $e_i>0$, $\u_i$ denotes a typical point of $\ga\cap U_i$,
and for each $i$ such that $e_i<a$, $\w_i$ denotes a typical point of 
$U_i\setminus(\ga\cap U_i)$.}

To simplify the analysis, replacing $\Ga$ by its complementary code $\ov{\ga}$
in $J(v,v-k)$ if necessary, we may assume that $k\leq v/2$ (see remark (d) in 
Subsection~\ref{rem-flag-tra}). Note that in Table~\ref{tbl-utype}, the codes in 
lines 2, 4 are complementary to codes in lines 1, 3 respectively, while for the other 
lines the complementary code belongs to the same line.

\medskip\noindent
{\bf Case $b=2$.}\quad Since $k\leq v/2=a$, if $e_1=a$, then $k=a$, 
$\ga=U_1$, and as $G$ is transitive on $\mathcal{U}$, we have 
$\Ga=\mathcal{U}$ as in Example~\ref{ex:imprim}, Line 1 of Table~\ref{tbl-utype},
and in Example~\ref{ex:imprim2} with $k_0=1$. 
Suppose then that $e_1<a$, and hence also $e_2<a$,
and set $\be_1=(\ga\setminus\{\u_1\})\cup \{\w_1\}$ and $\be_2=(\ga\setminus
\{\u_1\})\cup \{\w_2\}$. Then $\t_\U(\be_1)=\t_\U(\ga)$ while $\t_\U(\be_2)
=\{e_1-1,e_2+1\}$. 
Suppose first that $\t_\U(\be_2)=\t_\U(\ga)$. Then $k$ 
is odd, $\t_\U(\ga)=\{\frac{k+1}{2},\frac{k-1}{2}\}$, and $\be_3:=(\ga\setminus\{
\u_2\})\cup\{\w_1\}$ has $\t_\U(\be_3)\ne \t_\U(\ga)$. Thus in this case $\be_3\in \Ga_1$
and $\be_1, \be_2\in\Ga$ for all choices of the $\u_i, \w_i$, so $\Ga$ consists 
of all $k$-subsets of type $\t_\U(\ga)$, and $a\geq \frac{k+1}{2}+1\geq3$,
as in Example~\ref{ex:imprim},  Line 6 of Table~\ref{tbl-utype}. 
Now assume that $\t_\U(\be_2)\ne \t_\U(\ga)$ so that
$\be_2\in\Ga_1$. Then $\be_1\in\Ga$ for all choices of $\u_1,\w_1$. Suppose also
that $e_2>0$. Then $(\ga\setminus\{\u_2\})\cup\{\w_2\}\in\Ga$ for all choices of
$\u_2,\w_2$. Since 
$G$ is transitive on $\mathcal{U}$ it follows in this case that $\Ga$ consists of all
$k$-subsets of type $\t_\U(\ga)$. If $e_1=e_2$, then $\Ga$ is as in 
Example~\ref{ex:imprim},  Line 3 or 5 of Table~\ref{tbl-utype}. 
So assume that $e_1>e_2>0$. Then $\be_3:=(\ga\setminus\{
\u_2\})\cup\{\w_1\}$ has $\t_\U(\be_3)\ne \t_\U(\ga), \t_\U(\be_2)$, which is a 
contradiction. This leaves the possibility $e_2=0$, and here $\Ga$ is as in Example~\ref{ex:imprim}, Line 1 of 
Table~\ref{tbl-utype}.

\medskip\noindent
{\bf Case $a=2$, $b\geq3$.}\quad Here $\t_\U(\ga)=\{2^c,1^d\}$ for some $c,d$
such that $k=2c+d$. Suppose first that $c\geq1$ and $d\geq2$. Since $k\leq 
v/2=b$, we have $e_b=0$. 
Then $(\ga\setminus\{\u_1\})\cup\{\w_b\}$ has $\U$-type
$\{2^{c-1},1^{d+2}\}$, and $(\ga\setminus\{\u_{c+2}\})\cup\{\w_{c+1}\}$ 
has $\U$-type $\{2^{c+1},1^{d-2}\}$, so both lie in $\Ga_1$. Thus $G$
is not transitive on $\Ga_1$ and we have a contradiction. Therefore either $c=0$
or $d\leq 1$. Suppose first that $c=0$, and hence $d=k$. Then  
$\be:=(\ga\setminus\{\u_{2}\})\cup\{\w_{1}\} 
\not\in \Ga$, since it has $\U$-type $\{2,1^{k-2}\}$, and so $\be\in\Ga_1$.
Thus, for each $i\leq d$ and $j>d$, $(\ga\setminus\{\u_{i}\})\cup\{\w_{j}\}$ 
has $\U$-type $\t_\U(\ga)=\{1^k\}$ and hence lies in $\Ga$. 
It follows that $\Ga$ consists of all $k$-subsets of type 
$\{1^k\}$, as in Example~\ref{ex:imprim},  Line 3 of Table~\ref{tbl-utype}. 
Thus we may assume that $c>0$ and
$d\leq 1$. 

Next take $d=0$, so $c=k/2$. If $c=1$ then $\Ga$ is as in
Example~\ref{ex:imprim},  Line 1 of Table~\ref{tbl-utype}. so assume that $c\geq2$.
For $\ga'\in\Ga$ set $\ga_0(\ga'):=\{ U\in\mathcal{U}\,|\,U\subset \ga'\}$, and let
$\Ga_0:=\{\ga_0(\ga')| \ga'\in\Ga\}$. Then $\Ga_0\subset\binom{\U}{c}$ 
with $2\leq c\leq\frac{b}{2}$, so $\Ga=\Ga(2,\Ga_0)$, as in  Example~\ref{ex:imprim2}.

Finally suppose that $d=1$, so $c=(k-1)/2$. Let
$i\leq c$ and $j>c+1$ (note that $c+1<b$ since 
$k\leq\frac{v}{2}=b$). Then $\Ga_1$ contains
$\be:=(\ga\setminus\{\u_{i}\})\cup\{\w_{j}\}$ of $\U$-type $\{2^{c-1},1^3\}$, 
and since $\ga':=
(\ga\setminus\{\u_{c+1}\})\cup\{\w_j\}$ has $\U$-type $\{2^{c},1\}$,
it lies in $\Ga$. Letting $\w_j$ and $j$ vary, we deduce that $\Ga$ contains all 
$k$-subsets that contain $\ga\setminus\{\u_{c+1}\}=U_1\cup\dots,\cup U_c$. 
Also the $k$-subset $\ga'':=(\ga'\setminus\{\u_{i}\})
\cup\{\w_{j}'\}$, where $U_{j}=\{\w_{j},\w_{j}'\}$, has $\U$-type $\{2^{c},1\}$, 
and hence lies in $\Ga$. Applying the previous argument to $\ga''$ yields 
that $\Ga$ contains all $k$-subsets that contain $U_1\cup\dots\cup U_{i-1}\cup 
U_{i+1}\cup\dots\cup U_c\cup U_{j}$, and this holds for all $i\leq c<j$. 
It follows that $\Ga$ consists of all $k$-subsets with $\U$-type $\{2^c,1\}$ as in 
Example~\ref{ex:imprim}, Line 7 of Table~\ref{tbl-utype}.

\medskip\noindent
{\bf Case $a\geq3$ and $b\geq3$. } We divide this remaining case into several subcases. 

\medskip\noindent
{\bf Subcase $e_1-e_b\leq1$.} \quad 
Here $\t_\U(\ga)=\{c^i,(c-1)^{b-i}\}$, with 
$1\leq c\leq a$ and $1\leq i\leq b$, and 
$k=ci+(c-1)(b-i)\leq v/2$, so that $c<a$. Suppose first that both 
$i<b$ and $c\geq2$. Then $\be_1:=(\ga\setminus\{\u_{i+1}\})\cup\{\w_1\}$ has $\U$-type
$\t_\U(\be_1)=\{c+1,c^{i-1},(c-1)^{b-i-1},c-2\}\ne \t_\U(\ga)$, and hence $\be_1
\in\Ga_1$. If $i\geq2$, then $(\ga\setminus\{\u_{2}\})\cup\{\w_1\}$ has $\U$-type
$\{c+1,c^{i-2},(c-1)^{b-i+1}\}\ne \t_\U(\ga)$ or $\t_\U(\be_1)$ and we have a
contradiction. Thus $i=1$, but then $(\ga\setminus\{\u_{3}\})\cup\{\w_2\}$ 
has $\U$-type $\{c^{2},(c-1)^{b-3},c-2\}$ and again
we have a contradiction. Therefore either $i=b$ or $c=1$. Suppose first that $i=b$.
Then, for distinct $j,\ell$, the $k$-subset $\be_2:=(\ga\setminus\{\u_{j}\})
\cup\{\w_\ell\}$ has $\U$-type $\t_\U(\be_2)=\{c+1,c^{b-2},c-1\}$, so 
$\be_2\in\Ga_1$. This implies that $(\ga\setminus\{\u_{j}\})
\cup\{\w_j\}\in\Ga$, for all $j$ and all choices of the points $\u_j,
\w_j$, and it follows easily that $\Ga$ 
consists of all $k$-subsets of $\U$-type
$\{c^b\}$ as in Example~\ref{ex:imprim}, Line 3 or 5 of Table~\ref{tbl-utype}.
Suppose finally that $i<b$ and $c=1$. Again it is easy to see that
$k$-subsets in $\Ga_1$ have $\U$-type $\{2,1^{k-2}\}$ and that $\Ga$ consists of all 
$k$-subsets of $\U$-type $\{1^k\}$ as in Example~\ref{ex:imprim}, Line 3 of Table~\ref{tbl-utype}.

\medskip
Thus we may assume that $e_1\geq e_b+2$. Let $j$ be minimal such that 
$e_1\geq e_j+2$. Then $e_{j-1}=e_1$ or $e_1-1$, and $e_j\leq a-2$. Define 
$\be:=(\ga\setminus\{\u_1\})\cup\{\w_j\}$ with 
$\t_\U(\be)=\{e_1-1,e_2,\dots,e_{j-1},e_j+1,e_{j+1},\ldots,e_b\}$ (possibly with
the first and $j^{th}$ entries out of order). In particular $\t_\U(\be)\ne 
\t_\U(\ga)$, so $\be\in\Ga_1$. Note that there is no entry of 
either $\t_\U(\ga)$ or $\t_\U(\be)$ greater than $e_1$.

\medskip\noindent
{\bf Subcase $e_b+2\leq e_1<a$.}\quad If $e_2>0$, then 
$(\ga\setminus\{\u_2\})\cup\{\w_1\}$ has $\U$-type different from 
$\t_\U(\ga), \t_\U(\be)$, and we have a contradiction. 
Thus  $e_2=0$, so $j=2$, $k=e_1<a$, $\ga\subset U_1$, and so
$\t_\U(\be)=\{e_1-1,1\}$. Since $G$ is transitive on $\Ga$ and $\Ga_1$, 
it follows that $\Ga$ consists of all $k$-subsets of $\U$-type $\{k\}$, 
as in Example~\ref{ex:imprim}, Line 1 of Table~\ref{tbl-utype}. 

\medskip\noindent
{\bf Subcase $e_b+2\leq e_1=a$.}\quad
There exist $j_1\geq1, j_2\geq0$ such that $j_1+j_2=j-1$ and 
$\t_\U(\ga)=\{a^{j_1},(a-1)^{j_2},e_j,\dots,e_b\}$. Then 
$\t_\U(\be)=\{a^{j_1-1},(a-1)^{j_2+1},e_j+1,e_{j+1},\dots,e_b\}$. 
Also $\be':=(\ga\setminus\{\u_1\})\cup\{\w_b\}$
has $\U$-type $\{a^{j_1-1},(a-1)^{j_2+1},e_j,\dots,e_{b-1},e_b+1\}$,  
different from $\t_\U(\ga)$, and so $\be'\in\Ga_1$ and $\t_\U(\be')=
\t_\U(\be)$. This implies that $e_j=e_b=c$, say, and $c\leq a-2$, so that
$\t_\U(\ga)=\{a^{j_1},(a-1)^{j_2},c^{b-j+1}\}$ and 
$\t_\U(\be)=\{a^{j_1-1},(a-1)^{j_2+1},c+1,c^{b-j}\}$.

Suppose that  $j_2>0$. 
Then $\ga':=(\ga\setminus\{\u_{j_1+1}\})\cup\{\w_j\}$
has $\U$-type $\{a^{j_1},(a-1)^{j_2-1},a-2,c+1,c^{b-j}\}\ne \t_\U(\be)$,
and so $\ga'\in\Ga$. Thus $\t_\U(\ga')=\t_\U(\ga)$ and it follows that $c=a-2$.
If $a\geq4$ then $k>b\cdot\frac{a}{2}=\frac{v}{2}$, which is not so, and hence 
$a=3$. Thus $k=b+2j_1+j_2\leq \frac{v}{2}=\frac{3b}{2}$, so $2j_1+j_2\leq
\frac{b}{2}$. In particular, the number of entries of $\t_\U(\ga)$ equal 
to 1 is $b-j_1-j_2\geq \frac{b}{2}+j_1\geq2$. Hence the $k$-subset
$(\ga\setminus\{\u_b\})\cup\{\w_j\}$
has $\U$-type different from $\t_\U(\ga)$ and $\t_\U(\be)$, 
and this is a contradiction.

Thus $j_2=0$, so $\t_\U(\ga)=\{a^{j-1},c^{b-j+1}\}$ and 
$\t_\U(\be)=\{a^{j-2},a-1,c+1, c^{b-j}\}$. Since $k\leq \frac{v}{2}$ 
we have $\frac{ab}{2}\leq 
v-k=(a-c)(b-j+1)$. In particular, $b-j+1> b/2>1$.
If $c>0$ then $(\ga\setminus\{\u_b\})\cup\{\w_j\}$ has
$\U$-type $\{a^{j-1},c+1,c^{b-j-1},c-1\}$ different from $\t_\U(\ga)$ 
and $\t_\U(\be)$, and this is a contradiction. Thus $c=0$ and $\t_\U(\ga)=
\{a^{j-1}\}$ with $k=a(j-1)$. 
If $j=2$ then $\ga=U_1$, $k=a$, and $\Ga$ is as in 
Example~\ref{ex:imprim}, Line 1 of Table~\ref{tbl-utype}. 
If $j\geq3$, define $\Ga_0:=\{\ga_0(\ga')|\ga'\in\Ga\}$, where for
$\ga'\in\Ga$, $\ga_0(\ga')=\{U\in\mathcal{U}\,|\,U\subseteq\ga'\}$.
Since $2\leq j-1\leq b/2<b$, $\Ga=\Ga(a,\Ga_0)$, as in
Example~\ref{ex:imprim2}.
\end{proof}

The proof of Theorem~\ref{notprim} follows from Propositions~\ref{intrans} and~\ref{imprim}.

\section{Primitive neighbour-transitive codes}\label{sect:primitive}

%From Sections~\ref{sect:intrans} and~\ref{sect:imprim}, 
For our analysis of $G$-neighbour-transitive codes in $J(v,k)$ with $G\leq\Sym(\V)$,
it remains to consider codes $\Ga\subset\binom{\V}{k}$ such that 
$\Aut(\Ga)\cap\Sym(\V)$ acts primitively on $\V$. Our first task is to 
prove Theorem~\ref{flagtra}. We use the notation introduced in Subsection~\ref{notn}.

\begin{lemma}\label{lem:Delta}
   If $\Ga\subset\binom{\V}{k}$ and $\Aut(\Ga)\cap\Sym(\V)$ is transitive on $\V$, 
then, for $u\in\V$, the set $\Delta(\u)=\bigcap\{\ga'\in\Ga\,|\,\u\in\ga'\}$ is a 
block of imprimitivity for  $\Aut(\Ga)\cap\Sym(\V)$ in $\V$. In particular, if
$\Aut(\Ga)\cap\Sym(\V)$ is primitive on $\V$, 
then $\Delta(\u)=\{u\}$.
\end{lemma}

\begin{proof}
Let $G=\Aut(\Ga)\cap\Sym(\V)$ and $g\in G$. Then $\Delta(\u)^g=\Delta(\u^g)$, 
and hence, since $G$ is transitive on $\V$, 
$|\Delta(\u)|$ is independent of $\u$. Suppose that 
$g\in G$ and $\w\in\Delta(\u)\cap\Delta(\u)^g$. Now $\w\in\Delta(\u)$ 
implies that $\Delta(\u)\supseteq\Delta(\w)$, and consequently these two sets 
are equal. Similarly $\Delta(\w)=\Delta(\u^g)$, and hence $\Delta(\u)=
\Delta(\u)^g$. Thus $\Delta(\u)$ is a block of imprimitivity for the action of 
$G$ on $\V$. In particular, if $G$ is primitive on $\V$ then,
since $|\Delta(\u)|\leq k<v$, we conclude that 
$\Delta(\u)=\{\u\}$.   
\end{proof}

\subsection{Proof of Theorem~\ref{flagtra}}  Suppose that  
$\Ga\subset\binom{\V}{k}$,  $\ga\in\Ga$, $2\leq k\leq v-2$, 
and $G\leq \Aut(\Ga)\cap\Sym(\V)$. 
To prove part (a), suppose first that $G$ is incidence-transitive on $\Ga$ and $\de(\Ga)\geq2$. In particular $G$ is 
transitive on $\Ga$. Since $\de(\Ga)\geq2$, it follows 
that $J(\ga)\subseteq\Ga_1$, and hence $G_\ga$ is transitive on $J(\ga)$. 
This implies that $G_\ga$ is transitive on $\ga\times\ov\ga$. The converse assertion follows from Lemma~\ref{str-flagtra}.

(b)\quad Next suppose that $\de(\Ga)\geq3$ and $\Ga$ is $G$-neighbour transitive. By
part (a), it is sufficient to prove that $\Ga$ is $G$-incidence transitive. Let $(\ga,\ga_1), (\ga',\ga_1')\in\Ga\times\Ga_1$ 
be two incidences. Since $G$ is transitive on $\Ga_1$ there exists $g\in G$ such that $\ga_1^g=\ga_1'$. 
Then $\ga^g, \ga'$ both lie in $J(\ga_1')$ and hence $d(\ga^g,\ga')\leq2$. Since $\de(\Ga)\geq3$, this means that $\ga^g=\ga'$. Hence $g$ maps  $(\ga,\ga_1)$ to $(\ga',\ga_1')$, so $\Ga$ is $G$-incidence-transitive.

(c)\quad Finally assume that $G$ is primitive on $\V$ and that $\Ga$ is $G$-strongly incidence transitive. 
Let $\u\in\V$. Since $\ov{\Ga}$ is also $G$-strongly incidence-transitive (see Subsection~\ref{rem-flag-tra}),
we may assume that $2\leq k\leq v/2$. By Lemma~\ref{lem:Delta}, the set $\Delta(\u)=\bigcap\{\ga'\in\Ga\,|
\,\u\in\ga'\}$ is equal to $\{u\}$.  We use $G$-strong incidence-transitivity. 
Let $\w,\w'$ be distinct points of 
$\V\setminus\{\u\}$. Since $\Delta(\u)=\{\u\}$, there are codewords $\ga, \ga' 
\in \Ga$ containing $\u$  such that
$\w \not\in \ga$ and $\w' \not\in \ga'$. Since $k\le v/2$ and 
$\u \in \ga\cap\ga'$, it follows that $\ov\ga\cap\ov 
\ga'$ contains at least one point, $\v$ say. By strong incidence-transitivity, $G_{\ga,\u}$ 
 is transitive on $\ov\ga$  and $G_{\ga' ,\u}$ is transitive on $\ov\ga'$.
Hence there are elements $g\in G_{\ga, \u}$ and $g'\in G_{\ga', \u}$ such that $\w^g = \v$ 
and $\v^{g'} = \w'$. 
It follows that $gg'\in G_{\u}$ and $\w^{gg'}=\w'$, and hence that $G_\u$ is transitive on 
$\V\setminus\{\u\}$. Thus $G$ is 
2-transitive on $\V$, completing the proof.

\subsection{Organising the 2-transitive classification}\label{sect:organisation}

From now on we suppose that $\Ga\subset\binom{\V}{k}$ with $\delta(\Ga)\geq2$, 
where $2\leq k\leq |\V|-2=v-2$, and that $G\leq\Aut(\Ga)\cap\Sym(\V)$
acts strongly incidence-transitively on $\Ga$ and 2-transitively on $\V$. 
Since $\delta(\Ga)\geq2$,  $\Ga$ is a proper subset of $\binom{\V}{k}$, and in particular
the group $G$ is not $k$-homogeneous on $\V$, that is to say, 
$G$ is not transitive on the $k$-subsets of $\V$. In particular $3\leq k\leq v-3$, and 
$G$ does not contain the alternating group $A_v$.  Note that,  by Theorem~\ref{flagtra},
each $G$-neibour transitive code $\Ga$ with $\delta(\Ga)\geq3$ is $G$-strongly incidence-transitive. 

\medskip\noindent
\emph{Comments on the strategy:}\quad 
Those $2$-transitive groups $G$ which do not lie in an infinite family of 
2-transitive groups have been analysed completely in \cite{NP}. Thus we assume 
that $G$ lies in one of the infinite families 
of $2$-transitive groups, as listed in for example in \cite[Chapter 7.3 and 7.4]{Cam}
or \cite[Chapter 7.7]{DM}. As explained in the introduction, in this paper we address 
all families apart from the symplectic groups acting on quadratic forms. Thus
we investigate the following cases.
\begin{description}
   \item[affine]\quad $G\leq {\rm A}\GaL(\V)$ acting on $\V=\F_q^n$;
\item[linear] \quad $\PSL(n,q)\leq G\leq\PGaL(n,q)$ on $\PG(n-1,q)$;
\item[rank 1] \quad the Suzuki, Ree and Unitary groups.
\end{description}

We treat the various infinite families of 2-transitive groups $G$
separately. Let $\ga\in\Ga$. Since $\Ga$ is $G$-strongly incidence-transitive,  
$G_\ga$ is transitive on $\ga\times\ov\ga$. 
In particular $k(v-k)$ divides $|G_\ga|$, and $G$ is not $k$-homogeneous. 

For each of these $2$-transitive groups $G$, we need to determine all possibilities 
for the stabiliser $G_\ga$ (up to conjugacy). Note that, if $G_\ga\leq H < G$ 
and $H$ is intransitive on $\V$ then, since $G_\ga$ has only two orbits on $\V$, 
namely $\ga$ and $\ov\ga$, it follows that the $H$-orbits are the sets $\ga$ and 
$\ov\ga$, and hence $H=G_\ga$. Thus $G_\ga$ is a proper subgroup of $G$ which is
maximal subject to having two orbits in $\V$.   
We make a small observation about the case of a transitive subgroup $H$. 

\begin{lemma}\label{obs}
Suppose that $\Ga$ is $G$-strongly incidence-transitive with $G\leq\Aut(\Ga)\cap\Sym(\V)$. 
Let $\ga\in\Ga$, and suppose that $G_\ga < H < G$ with $H$ transitive 
on $\V$ and leaving invariant a non-trivial partition $\Pi$ of $\V$. 
Then $\ga$ is a union of some of the blocks of $\Pi$. 
% of $G$ containing $G_\ga$ and $H$ is 
% transitive on $\V$. Suppose also that $N$ is a normal subgroup of $H$ which is intransitive 
% on $\V$. Then $\ga$ is a union of $N$-orbits.
\end{lemma}

\begin{proof}
Let $\pi\in\Pi$ be a block of $\Pi$ containing a point $\u$ of $\ga$. 
We claim that $\pi\subseteq\ga$. Suppose to the contrary that 
$\pi\cap\ov{\ga}$ contains a point $\w$. Then $G_{\ga,\u}\leq 
H_{\u} < H_\pi$, so $\pi$ contains the $G_{\ga,\u}$-orbit 
containing $\w$, namely $\ov{\ga}$. This implies that $G_\ga$ 
fixes the block $\pi$ setwise, so $\pi$ also contains the $G_\ga$-orbit 
containing $\u$, namely $\ga$. Thus $\pi=\V$, a contradiction.
Hence $\ga\subseteq\pi$.  
\end{proof}

\section{Affine groups}\label{sec-aff}

In this section we treat the 2-transitive affine groups. Here $\V =\F_q^n$ is an $n$-dimensional 
vector space over a field $\F_q$ of order $q=p^a$ and $v=q^n$, where 
$p$ is a prime and $a, n\geq1$. The group $G$ is a semidirect product 
$N.L$, where $N$ is the group of translations of $\V$ and $L$ is a 
subgroup of the group $\GaL(\V)$ of semilinear transformations of $\V$, 
which is transitive on $\V^\#=\V\setminus\{0\}$. So $G$ is a subgroup of 
$X=\AGaL(\V)$, the full affine semilinear group. We view $\V$ as the point 
set of the affine geometry ${\AG}(n,q)$. We use the notation introduced in 
Section~\ref{sect:organisation}.

\subsection{One-dimensional affine groups}\label{sub-aff1}

Here $n=1$ and we identify $\V$ with $\F_q$. For application in 
the case of arbitrary dimension, we only assume in this subsection 
that $1\leq k\leq v-1$, and we do not insist that 
$\Ga$ is a proper subset of $\binom{\V}{k}$. 
Let $x$ be a primitive element of $\F_q$, so that $A:=\la x\ra$
is the multiplicative group of $\F_q$. Let $B:=\la \si\ra= \Aut(\F_q)
\cong Z_a$. Then $N\cong Z_p^a$ and $X=N.(A.B)=\AGaL(1,q)$.
 The aim of this section is to prove Proposition~\ref{aff1}.

%Since $X$ is 2-transitive on $\Om$ we may assume, by 
%replacing $G$ by an $X$-conjugate if necessary, that 
 %$\Om_1$ contains both $0$ and $1\in F_q$.

\begin{proposition}\label{aff1}
If $n=1$, then one of the following holds.
\begin{itemize}
\item[(i)] $k=1$ or $q-1$; 
\item[(ii)] $v=q=4$, $k=2$, $G_\ga= \AGL(1,2).2\cong Z_2^2$, 
$\ga$ or $\ov{\ga}$ is $\F_2$, and $\delta(\Ga)=1$;
\item[(iii)] $v=q=16$, $k=4$ or $12$, $G_\ga= [2^2].\GL(1,4).[4]$, 
$\ga$ or $\ov{\ga}$ is $\F_4$ (as $k$ is $4$ or $12$ respectively), and $\delta(\Ga)=3$.
\end{itemize}
\end{proposition}

We note that only the example
in Proposition~\ref{aff1}~(iii) yields a strongly incidence-transitive code, since in case (ii), $\Ga=\binom{\V}{2}$.

Set $M:=G_\ga\cap N$. A \emph{primitive prime divisor} of $p^a-1$ 
is a prime divisor $r$ of $p^a-1$ such that $r$ does not divide 
$p^i-1$ for any positive integer $i<a$. For such 
a prime $r$, $p$ has order $a$ modulo $r$ and so $a$ divides $r-1$.
In particular $r\geq a+1$. By \cite{Z}, such primes exist unless 
$(p,a)=(6,2)$, or $a=2$ and $p=2^b-1$ for some $b$.

\begin{lemma}\label{k1} 
The parameter $k\in\{1,v-1\}$ if and only if $M=1$. 
\end{lemma}

\begin{proof}
If $k=1$ or $v-1$ then $G_\ga$ fixes a non-zero  element of $\V$ and so $M=1$. 
Conversely suppose that $M=1$, and suppose that $2\leq k\leq v-2=q-2$. In particular $q\geq 4$.  
Then $|G_\ga|$ is divisible by $k(v-k)=k(q-k)
\geq 2(q-2)$. Also $G_\ga\cong G_\ga N/N\leq  X/N\cong AB$ so $|G_\ga|$ divides
$(p^a-1)a$. In particular 
\[
2(p^a-2)\leq k(p^a-k)\leq |G_\ga|\leq (p^a-1)a.
\] 
If $a=1$, then these inequalities imply that $q\leq3$ which is a contradiction. Thus $a\geq2$. If $q=4$ then $k=2$, but then $k(q-k)=4$ does not divide $|G_\ga|$. Hence $q\geq9$. If $a=2$ then the displayed inequalities imply that $k\in\{2,q-2\}$, but then  $k(q-k)$ does not divide $|G_\ga|$. 
Hence $a\geq3$. Next if $q=64$, then $k(64-k)$ divides $|G_\ga|$ 
which divides $63.6$, but there is no such $k\in[2,62]$.
Thus $q\ne64$ and hence there exists a primitive prime divisor $r$ of $p^a-1$, and as we observed above, $r\geq a+1>3$.  Suppose that $r$ does not divide $|G_\ga|$.
Then $2(p^a-2)\leq (p^a-1)a/r<p^a-1$ which is a contradiction.
Thus $r$ divides $|G_\ga|$. A Sylow $r$ subgroup of $AB$ is contained 
in $A$ (since $r>a$) and hence is normal in $AB$. 
It follows that $G_\ga$ has a 
normal Sylow $r$-subgroup, say $R$. Without loss of generality we 
may assume that $R\leq A$. In particular $R$ has a unique 
fixed point in $\V$ which therefore must be fixed by $G_\ga$. This contradicts $2\leq k\leq v-2$.
\end{proof}

Now we suppose that $M\ne1$.  Let 
$K=\F_p[G_\ga\cap A]$ denote the subfield of $\F_q$ generated by $G_\ga\cap A$.

\begin{lemma}\label{kvecspace} 
If $M\ne 1$ then $M$ is a $K$-vector space and 
$K$ is a proper subfield of $\F_q$.
\end{lemma}

\begin{proof}
The group $M$ acts on $\V$ as a subgroup of translations. Thus for  some $Y\subseteq \F_q$,  
$M=\{t_y \,|\, y\in Y\}$, where $t_y:v\mapsto v+y$ for $v\in\V$. 
Interchanging $\ga$ and $\ov{\ga}$ if necessary, we may assume that 
$0\in\ga$. Then the $M$-orbit 
$0^M$ is equal to $Y$ and is contained in $\ga$. 
As $M$ is a subgroup of $N$, the set $Y$ is 
an $\F_p$-subspace of $\V$ (viewed as $\F_p^a$). Also $G_\ga\cap A$ normalises
$M$, and as $A$ acts by multiplication on $\V$ it follows that,
for each $z\in G_\ga\cap A$ and $t_y\in M$, we have $z^{-1}t_yz=t_{yz}\in M$,
that is to say $Y(G_\ga\cap A)\subseteq Y$. Since $K=\F_p[G_\ga\cap A]$ it follows that
$Y$ is also closed under multiplication by arbitrary elements of $K$.
Thus $Y$ is a $K$-vector space, that we identify with $M$. Since
$M$ is intransitive on $\V$, $K$ is a proper subfield of $\F_q$.
\end{proof}

\begin{lemma}\label{aff1ex} 
If $M\ne 1$ then (ii) or (iii) of Proposition~{\rm\ref{aff1}} holds.
\end{lemma}

\begin{proof}
Interchanging $\ga$ and $\ov{\ga}$ if necessary, we may assume that 
$0\in\ga$.
Now we use arithmetic. We have $|M|=p^s$, where $1\leq s<a$ by Lemma~\ref{kvecspace}.  
As $M$ is semiregular on $\V$,
$k=|\ga|=p^sm$ and $|\bar\ga|=p^a-p^sm$, for some $m$ such that $1\leq m< p^{a-s}$. By assumption
$p^sm(p^a-p^sm)$ divides $|G_\ga|$ which divides $p^s(p^a-1)a$.
Thus $p^sm(p^{a-s}-m)$ divides $|G_\ga|/p^s$ which divides $(p^a-1)a$.
In particular $p^s$ divides $a$, so $a\geq p\geq2$.

Suppose that $a=2$. Then also $p^s=2$, so $q=4$, $k=2$,
$G_\ga\cap A=1$, $G_\ga=M B\cong Z_2^2$, and hence $\ga=\{0,1\}=\F_2$, 
$\delta(\Ga)=1$, and  (ii) of 
Proposition~{\rm\ref{aff1}} holds. If $p^a=2^6$ then since $p^s$ divides $a=6$, it 
follows that $p^s=2$ and so $|M|=2$. Since $M$ is $G_\ga$-invariant we have $G_\ga\cap A=1$. Thus 
$m(64-2m)$ divides $|G_\ga|/2$ which divides $6$, a contradiction.
Hence we may assume that $a\geq3$ and $p^a\ne 64$. 

In particular 
$p^a-1$ has a primitive prime divisor, say $r$, 
%where $r\equiv 1\pmod{a}$. 
and as noted above, $r\geq a+1$. If $|G_\ga\cap A|$ were divisible by $r$ then the subfield $K$ would be equal to $\F_q$, contradicting Lemma~\ref{kvecspace}. Hence $r$ does not divide $|G_\ga\cap A|$ and so
$m(p^a-p^sm)\leq |G_\ga|/p^s\leq\frac{p^a-1}{r}a<p^a-1$.  
The function $x(p^a-p^sx)$ is continuous on the
interval $(1,p^{a-s})$ with a maximum at $x=p^{a-s}/2$. 
If $2\leq m\leq p^{a-s}-2$, then $p^a-1>m(p^a-p^sm)\geq 2(p^a-2p^s)$
implying that $p^a<4p^s-1<4p^s$. Hence $p^{a-s}<4$ so that $s=a-1$, which contradicts the fact that $p^s=p^{a-1}$ divides $a$ with $a\geq3$. Thus $m=1$ or $p^{a-s}-1$,
and so $\{|\ga|,|\bar\ga|\}=\{p^s,p^{a}-p^s\}$, and $p^{a-s}-1$ divides $\frac{p^a-1}{r}\frac{a}{p^s}$.

If $a=3$ then its divisor $p^s$ is also equal to 3, and the divisibility 
condition is that $p^{a-s}-1=8$ divides $26/r$, which is impossible. 
If $a=4$ then $p^s=2$ or 4, and the divisibility condition is `7 divides $30/r$' or 
`3 divides $15/r$' respectively. It follows that $p^s=4$, $q=16$, and $12$ divides 
$|G_\ga|/4$, so that $G_\ga = M(\la x^5\ra .B)$. Hence the subfield $K=\F_4$ and 
the two orbits of $G_\ga$ in $\V$ are $\F_4$ (which must equal $\ga$ since $0\in
\ga$) and its complement. Moreover $G_{\ga,0}= \la x^5\ra .B$ is transitive on 
$\bar\ga=\V\setminus\F_4$, so we have an example. Now $G_\ga$ induces a 2-transitive action on $\ga$, 
and the stabiliser $G_{\ga,\{0,1\}}$ has order $8$. It follows that $G_{\ga,\{0,1\}}= G_{\{0,1\}}$,
and hence that $\ga$ is the only `codeword' containing $\{0,1\}$, so 
$\delta(\Ga)=3$ and (iii) of Proposition~\ref{aff1} holds. 

Thus we may assume that $a\geq5$. The facts that $p^s$ divides $a$ and $a\geq5$ together
imply that $a-s\geq4$. Hence either (i) $p=2, a=8, s=2$ and $p^{a-s}-1=2^6-1=63$, or (ii) $p^{a-s}-1$ has 
a primitive prime divisor $r'$. Case (i) is impossible since $63$ does not divide $a(p^a-1)$. Thus we are in case (ii) and the prime $r'$ divides $\frac{p^a-1}{r}\frac{a}{p^s}$. Suppose first that $r'$ divides $a/p^s$. Then $a\geq r'p^s\geq (a-s+1)p^s$, and hence 
\[
a\leq\frac{p^s(s-1)}{p^s-1}=(s-1)(1+\frac{1}{p^s-1})\leq 2(s-1)<2^s\leq p^s\leq a 
\]
which is a contradiction. Hence $r'$ divides $p^a-1$. This implies that $a-s$ divides $a$, and hence $a-s\leq a/2$ so $s\geq a/2\geq p^s/2\geq 2^{s-1}$. It follows that either $a=p=2s=2$ or $a/2=p=s=2$, contradicting the assumption that $a\geq5$. 
\end{proof}

Proposition~\ref{aff1} follows from Lemmas~\ref{k1} and \ref{aff1ex}. 

\subsection{General affine case}\label{sub-aff2}

Now suppose that $G=N.L\leq X=\AGaL(n,q)$ with $q=p^a$ and $n\geq2$, 
acting on $\V=\F_q^n$ and that $G_\ga$ is transitive on $\ga\times\bar\ga$. 
The affine subspaces and their complements provide natural families of 
examples, since taking $G=X$ and $\ga$ or $\ov{\ga}$ an affine subspace, 
the group $G_\ga$ is transitive on $\ga\times\ov{\ga}$.

\begin{example}\label{ex-aff2}
{\rm 
For any positive integer $s<n$, the set $\Ga$ of affine $s$-dimensional 
subspaces, and the set $\ov{\Ga}$ of complements of these $s$-subspaces, 
are $X$-xtrongly incidence-transitive codes. 
}
\end{example}

Our main result for the affine case shows that examples apart from 
those in Example~\ref{ex-aff2} are very restricted. In particular, 
the codeword $\ga$ or its complement is a subset of class $[0,\sqrt{q},q]_1$
(as defined before Theorem~\ref{thm-al}) and $q$ must be 4 or 16.

\begin{proposition}\label{aff2}
Suppose that $\V=\F_q^n$ with $n\geq2$, and $\Ga\subset\binom{\V}{k}$
is $G$-strongly incidence-transitive, where $G\leq\AGaL(n,q)$ is $2$-transitive on $\V$. 
Let $\ga\in\Ga$. Then one of the following holds.
\begin{enumerate}
 \item[(i)] $\ga$ or $\ov{\ga}$ is an affine subspace as in Example~{\rm\ref{ex-aff2}}, or
\item[(ii)] $q=4$ and  each line of $\AG(n,4)$ either 
lies in $\ga$ or $\ov\ga$, or intersects $\ga$ in a Baer sub-line. Moreover 
$\frac{q^n+2}{3}\leq k\leq\frac{2(q^n-1)}{3}$.
\item[(iii)] $q=16$ and, interchanging $\ga$ and $\ov\ga$ if necessary, 
each line of $\AG(n,q)$ either lies in $\ga$ or $\ov\ga$, or intersects 
$\ga$ in a Baer sub-line of size $4$. Moreover 
$\frac{q^n+4}{5}\leq k\leq\frac{4(q^n-1)}{15}$.
\end{enumerate}
\end{proposition}

\begin{proof}
Since $G_\ga$ is transitive on $\ga\times\ov\ga$, it follows that 
$G_\ga$ is transitive on the set $\calL$ of lines of the affine space
$\AG(n,q)$ that meet both $\ga$ and $\bar\ga$. Thus for $\lambda\in\calL$,
$|\ga\cap\lam|=x$ is independent of the choice of $\lam$, and 
$|\bar\ga\cap\lam|=q-x$ with $1\leq x<q$. 
% Interchanging $\ga$ and 
% $\bar\ga$ if necessary we may assume that $x\leq q/2$, and we note 
% that we may not now have $k\leq v/2$.

Moreover, the group induced on $\lam$ by $G_{\ga,\lam}$ is a subgroup of
$\AGaL(1,q)$. Let $\v\in\ga\cap\lam$. Then $G_{\ga,\v}$ is transitive
on $\bar\ga$, and moreover the subset of lines of $\calL$ containing $\v$
induces a $G_{\ga,\v}$-invariant partition of $\bar\ga$ into parts of size $q-x$.  
Hence $G_{\ga,\v,\lam}$ is transitive on $\bar\ga\cap\lam$, and similarly, for
 $\u\in\bar\ga\cap\lam$, 
$G_{\ga,\u,\lam}$ is transitive on $\ga\cap\lam$. Thus the
subgroup of $\AGaL(1,q)$ induced by $G_{\ga,\lam}$ on $\lam$
is transitive on $(\ga\cap\lam)\times(\bar\ga\cap\lam)$. It follows from
Proposition~\ref{aff1} that one of $x\in\{1,q-1\}$, or 
$q=4$ with $x=2$, or $q=16$ with $x\in\{4, 12\}$. 

Suppose first that $x\in\{1, q-1\}$. Interchanging $\ga$ and $\ov\ga$ 
if necessary, we may assume that $x=1$. Then, for any pair of distinct points
$\v, \v'\in\ga$, the unique line $\lambda$ containing $\v$ and $\v'$
lies entirely within $\ga$. Thus $\ga$ (or the original $\ov\ga$) 
is an affine subspace of $\AG(n,q)$, as in Example~\ref{ex-aff2}.

Now we consider the other possibilities. Interchanging $\ga$ and $\ov\ga$ 
if necessary, we may assume that $q=x^2\in\{4, 16\}$. Then the subset $\calL'$ 
of $\calL$ consisting of lines containing a given point $\v\in\ga$ induces a 
partition of $\bar\ga$ with $(q^n-k)/(q-x)$ parts of size $q-x$, and $G_{\ga,\v}$ 
is transitive on $\calL'$. Each line of $\calL'$ intersects $\ga$ in a subset of 
size $x$ which, by Proposition~\ref{aff1} is a Baer sub-line. The additional 
$(x-1)(q^n-k)/(q-x)$ points of $\ga$ lying on 
these lines (apart from $\v$) forms a $G_{\ga,\v}$-orbit. Thus $k\geq 1+(x-1)(q^n-k)/(q-x)$,
and rearranging gives $k\geq \frac{(x-1)q^n+q-x}{q-1}$ which, since $q=x^2$, 
gives $k\geq \frac{q^n+x}{x+1}$. 

Similarly,  the subset $\calL''$ of $\calL$ consisting of lines containing 
a given point $\v'\in\bar\ga$ induces a partition of $\ga$ with 
$k/x$ parts of size $x$, and $G_{\ga,\v'}$ is transitive on $\calL''$. 
Each line of $\calL''$ intersects $\bar\ga$ in a subset of size $q-x$. 
The additional $(q-x-1)k/x$ points of $\bar\ga$ lying on 
these lines (apart from $\v'$) form a $G_{\ga,\v'}$-orbit. 
Thus $|\ov{\ga}|=q^n-k\geq 1+(q-x-1)k/x$, and rearranging gives $k\leq \frac{x(q^n-1)}{q-1}$.
For $q=4$ we therefore have $\frac{q^n+2}{3}\leq k\leq \frac{2(q^n-1)}{3}$ and for $q=16$ we have 
$\frac{q^n+4}{5}\leq k\leq \frac{4(q^n-1)}{15}$, and parts (ii) and (iii) hold.
\end{proof}

We note that the $2$-transitive hyperoval in $\PG(2,4)$ provides an 
example for case (ii) of Proposition~\ref{aff2}.

\begin{example}\label{ex-aff4}{\rm
Let $\ga$ be a $2$-transitive hyperoval in the projective
plane $\PG(2,4)$, and let $\lam$ be an external line of $\ga$.
Then $k=|\ga|=6$ and
the complement of $\lam$ in the point set of $\PG(2,4)$
is an affine space $\V=\AG(2,4)$ containing $\ga$. Let $\bar\ga=
\V\setminus\ga$. Then the subgroup $G_\ga$ of $\PGaL(3,4)$
stabilising $\lam$ and $\ga$ setwise acts faithfully on $\V$
and is transitive on $\ga\times\bar\ga$.

To see this observe that $G_\ga=\SL(2,4)\la\si\ra\cong S_5$
is $2$-transitive on $\ga$, and for $\v\in\ga$, 
$G_{\ga,\v}$ acts transitively on the five secant
lines to $\ga$ containing $\v$. Each of these 
secants contains two points of $\bar\ga$ and one point of $\lam$.
Thus $G_{\ga,\v}$ is transitive on $\bar\ga$. 
}
\end{example}

\section{Linear case}\label{sec-linear}

In this section we investigate the 2-transitive projective linear groups. 
Here $\V$ is the point set of the projective geometry $\PG(n-1,q)$ of 
rank $n-1\geq1$ over a field $\F_q$ of order $q=p^a$, and $v= 
(q^n-1)/(q-1)$, where $p$ is a prime and $a\geq1$. Since the situation for affine 2-transitive 
groups was analysed in Section~\ref{sec-aff}, 
we assume that $(n,q)\ne (2,2)$ or $(2,3)$.
In general the group $G$ satisfies $\PSL(n,q)\leq G\leq X:=\PGaL(n,q)$ 
and we often assume that $G=X$.  We use the notation introduced in 
Section~\ref{sect:organisation}. Since $G$ is $2$-transitive and $\Ga$ is a proper subset of 
$\binom{\V}{k}$, we have in particular, $3\leq k\leq v-3$ and, 
for some $k$-subset $\ga\subset\V$, the stabiliser $G_\ga$ is transitive 
on $\ga\times\bar\ga$.

%-------------------------------------

\subsection{Rank 1 case}\label{sub-linrank1}
Here $n=2$, $q\geq4$, and we identify $\V=\PG(1,q)$ with $\F_q\cup\{\infty\}$. We show that all examples arise from Baer sub-lines of $\V$. If $q=q_0^2$ then the subset $\F_{q_0}\cup\{\infty\}$, and its $X$-translates are the \emph{Baer sub-lines}.

\begin{example}\label{ex-lin-2dim}{\rm
Let $q=q_0^2$, and let $\ga=\F_{q_0}\cup\{\infty\}$, 
%a Baer sub-line of $\V=\PG(1,q)$, 
so $k=|\ga|=q_0+1$. Then the group $X_{\ga}=N_X(\PGL(2,q_0))$
is transitive on $\ga\times\bar\ga$. Moreover, since any pair 
of Baer sub-lines intersects in at most one point, it follows 
that the corresponding strongly incidence-transitive code $\Ga$ 
has minimum distance $\de(\Ga)=q_0$.
}
\end{example}

\begin{proposition}\label{lem-2dim}
If $n=2$ and $3\leq k\leq q-2$, then $q=q_0^2$, $k=q_0+1$ or $q-q_0$, and $\ga$ or $\ov\ga$ 
is a Baer sub-line, as in Example~{\rm\ref{ex-lin-2dim}}.
\end{proposition}

\begin{proof}
We use the classification of the subgroups of $S:=\PSL(2,q)$, \cite[Chapter VII]{Dickson}. 
Replacing $\ga$ by $\ov\ga$ if necessary, we may assume that $3\leq k\leq v/2=(q+1)/2$.
It follows in particular that $G_\ga$ is not contained in a maximal parabolic subgroup, 
and $G_\ga\cap {\rm PGL}(2,q)\not\leq D_{2(q-1)}$. Suppose  that $G_\ga\cap {\rm PGL}(2,q)\leq D_{2(q+1)}$. Since the Frobenius automorphism fixes the identity, 
the only way that $G_\ga$ can have two orbits is if
$k=(q+1)/2$. Then $(q+1)^2/4$ divides $|G_\ga|$, 
which divides $(q+1)a$. There 
are no possibilities with $q\geq4$. 

Now suppose that $G_\ga\leq N_X(\PSL(2,q_0))$, where $q=q_0^b$ for some $b>1$. 
Then $G_\ga$ fixes setwise a sub-line $\PG(1,q_0)$ of size $q_0+1<v/2$.
Since the only $G_\ga$-orbit length at most $v/2$ is $k$, we have
that $\ga$ is the set of points on a sub-line
${\rm PG}(1,q_0)$ with $k=q_0+1$, and hence $(q_0+1)(q-q_0)$ divides 
$|G_\ga|$, which divides $q_0(q_0^2-1)a$. Thus $q_0^{b-1}-1$ divides $(q_0-1)a$, which implies 
$b=2$. Thus $\ga$ is a Baer sub-line, as in Example~\ref{ex-lin-2dim}.

The remaining cases are those where $H:=G_\ga\cap{\rm PGL}(2,q)=A_5, S_4$ 
or $A_4$, and $H$ is not contained in any `subfield' subgroup $\PGL(2,q_0)$ with $q_0<q$. 
Suppose first that $H=A_5$. Then $q\equiv \pm1\pmod{10}$, and 
since $G_\ga$ is not contained in a subfield subgroup, $a\leq2$.
Since $3(q-2)\leq k(v-k)\leq |G_\ga|\leq 60a$, we have $q\in \{9, 11, 19\}$.
The cases $q=9$ and $q=11$ are not possible since in these cases $A_5$ is transitive 
on $\V$. Thus $q=19$. However there is no value of $k$ in the interval $[3,9]$ such that $k(20-k)$ divides $|G_\ga|=60$.

This leaves $A_4\leq G_\ga\leq S_4.\la\si\ra$ with $q$ odd. Since $G_\ga$ is not contained in a subfield subgroup, either $q=p$, or $q=p^2$ with $p\equiv \pm 3\pmod{8}$.
Since $3(q-2)\leq k(v-k)\leq |G_\ga|\leq 24a$, we have $q\in \{5,7,9\}$, and since $k(v-k)$ divides $|G_\ga|$, it follows that $q=9$ and $k=4$. However in this case $G_\ga$ is contained in $N_X(\PSL(2,3))$ which we assumed was not the case.
\end{proof}

%--------------------------------

\subsection{Higher rank linear case}\label{sub-lin2}

Now we assume that $n\geq3$. Here we have a family of examples arising 
from subspaces and their complements. 

\begin{example}\label{ex-lin1}
{\rm Let $1\leq s<n$ and let $\ga=\PG(s-1,q)$ be an $(s-1)$-dimensional subspace of  
$\V$, so $k=|\ga|=(q^s-1)/(q-1)$. Then the subgroup $X_{\ga}$ is transitive 
on $\ga\times\ov \ga$. Thus the set $\Ga$ of 
$(s-1)$-dimensional subspaces, and the set of their complements, 
form $X$-strongly incidence-transitive 
codes and each has minimum distance $\de(\Ga) =q^{s-1}$. }
\end{example}

For $u\in\ga, w\in\bar\ga$, we call the line $\lam(u,w)$ containing $u$ and 
$w$  a \emph{$\ga$-shared line}. Since $G_\ga$ is transitive on $\ga\times
\ov\ga$, the $\ga$-shared lines form a single $G_\ga$-orbit on lines, and 
in particular they all meet $\ga$ in a constant number $x$ of points, 
where $1\leq x\leq q$. Thus $\ga$ is a subset of class $[0,x,q+1]_1$. 
In Example~\ref{ex-lin1},  $x=1$.  

\begin{proposition}\label{lem-lin2}
Suppose that $\V=\PG(n-1,q)$, $\PSL(n,q)\leq G\leq\PGaL(n,q)$, 
and $\Ga\subset\binom{\V}{k}$ is $G$-strongly incidence-transitive, where 
$n\geq3$ and $3\leq k\leq |\V|-3$. 
Let $\ga\in\Ga$. Then  either $\ga$ or $\ov\ga$ is a projective subspace 
as in Example~{\rm\ref{ex-lin1}} or, interchanging $\ga$ and $\bar\ga$ if necessary,
$\ga$ is a subset of class $[0,x,q+1]_1$, where 
one of the following holds. 
\begin{enumerate}
\item[(a)] $x=2$ and $\frac{v-1}{q}+1\leq k\leq \frac{2(v-1)}{q}$, or
\item[(b)] $x=q_0+1$, $q=q_0^2$, $\frac{v-1}{q_0} +1\leq k\leq \frac{v-1}{q_0}+\frac{v-1}{q}$ and, for  each $\ga$-shared line $\lam$,  $\lam\cap \ga$ is a Baer sub-line.
\end{enumerate}
\end{proposition}

%________________Rewrite remark and Q

\begin{remark}\label{rem-lin2}{\rm
The parameters in part (a) suggest that $\ga$ might be a configuration similar 
to an oval or hyperoval in $\PG(2,q)$. For example, in $\PG(2,4)$,
the stabiliser of a hyperoval $\ga$ is $G_\ga\cong S_5$, transitive
on both $\ga$ and the complement $\bar\ga$ with $|\ga|=6, |\bar\ga|=15$. However,
for $u\in\ga$, $G_{\ga,u}$ has two orbits in $\bar\ga$, one of 
them an external line to $\bar\ga$. Thus this does not give rise to a strongly incidence-transitive code.

This argument about $\PG(2,4)$, together with Proposition~\ref{lem-lin2}, were used by Nico Durante 
in~\cite[Theorem 3.3]{Dur} to prove that there are no additional examples in the linear case 
satisfying the conditions of Proposition~\ref{lem-lin2}~(a) or~(b).
% Again, one might hope that examples in case (a) arise from partial spreads, 
% but they seem not to: take $n$ even, let $U,U'$ be disjoint $\frac{n}{2}$-dimensional subspaces  
% of the underlying space $V(n,q)$, and let $\ga$ be the set of 1-spaces (points of $\PG(n-1,q)$) contained in either $U$ or $U'$. In $\PGaL(n,q)$ the stabiliser of $\ga$ (of type $\GL(n/2,q)\wr S_2$) is transitive on $\ga$ and on $\bar\ga$, but unfortunately it has two orbits in $\ga\times\bar\ga$. (Each $w\in\bar\ga$ corresponds to a vector in $V(n,q)$ of the form $u+u'$ with $u\in U$ and $u'\in U'$ both non-zero; the elements of  $\ga\times\bar\ga$ correspond to pairs $(u'', u+u')$ with $u''\in U$ nonzero, and $u''$ may be linearly dependent or independent of $u$.) 
% 
% Similarly, one might hope that Baer sub-geometries  provide examples in case (b), but they do not for ranks higher than 1. For if $n\geq3$, $q=q_0^2$, and $\ga=\PG(n-1,q_0)$, 
% a Baer sub-geometry  of $\V$, then all lines meeting $\ga$ are shared lines, but the lines of the Baer geometry meet $\ga$ in $q_0+1$ points, while lines which meet $\ga$ and do not belong to the Baer sub-geometry meet $\ga$ in a single point. 
}
\end{remark}

\begin{proof}
The group $G_\ga$ is transitive on the set $\calL$ of $\ga$-shared lines 
and, for $\lambda\in\calL$, the group induced on $\lam$ by $G_{\ga,\lam}$ is a subgroup of
$\PGaL(2,q)$, independent of the choice of $\lam$. Let $\u\in\lam\cap\ga$. Then $G_{\ga,\u}$ is transitive
on $\bar\ga$ and moreover the subset of lines of $\calL$ containing $\u$
induces a $G_{\ga,\u}$-invariant partition of $\bar\ga$ with parts of size $q+1-x$.  
Hence $G_{\ga,\u,\lam}$ is transitive on $\lam\cap\bar\ga$. Similarly, if
 $\u\in\lam\cap\bar\ga$, then 
$G_{\ga,\u,\lam}$ is transitive on $\lam\cap\ga$. Thus the
subgroup of $\PGaL(2,q)$ induced by $G_{\ga,\lam}$ on $\lam$
is transitive on $(\lam\cap\ga)\times(\lam\cap\bar\ga)$. It follows from
Proposition~\ref{lem-2dim} that, interchanging $\ga$ and $\bar\ga$ 
if necessary,  $(x,q+1-x)=(1,q)$, $(2,q-1)$, or $(q_0+1,q-q_0)$, 
where in the third case, $q=q_0^2$ and $\lam\cap\ga$ is a Baer sub-line of 
$\lam$. 

Suppose first that $x=1$. Then, for any pair of distinct points
$\u, \u'\in\ga$, the  line $\lambda$ containing $\u$ and $\u'$
lies entirely within $\ga$. Thus $\ga$ is a subspace of $\V$ as in Example~\ref{ex-lin1}.

Now suppose that $x=2$ or $x=q_0$. 
Then the subset $\calL'$ of $\calL$ consisting of lines 
containing a fixed point $\u\in\ga$ induces a partition of $\bar\ga$ with 
$(v-k)/(q+1-x)$ parts of size $q+1-x$, and $G_{\ga,\u}$ is transitive on $\calL'$. 
Each line of $\calL'$ intersects $\ga$ in a set consisting of $\u$ 
and $x-1$ further points. The  $(x-1)(v-k)/(q+1-x)$ points of $\ga$, 
distinct from $\u$, lying on these lines forms a $G_{\ga,\u}$-orbit
contained in $\ga\setminus\{\u\}$. Thus $k\geq 1+\frac{(x-1)(v-k)}{q+1-x}$,
and hence $k\geq \frac{(x-1)(v-1)}{q} +1$. 

Similarly, if $\u\in\bar\ga$, then the subset $\calL''$ of $\calL$  consisting of lines 
containing $\u$ induces a partition of $\ga$ with 
$k/x$ parts of size $x$, and $G_{\ga,\u}$ is transitive on $\calL''$. 
Each line of $\calL''$ intersects $\bar\ga$ in a set consisting of $\u$ 
and $q-x$ further points. The  $(q-x)k/x$ points of $\bar\ga$, distinct from $\u$, 
lying on these lines forms a $G_{\ga,\u}$-orbit. Thus $v-k\geq 1+\frac{(q-x)k}{x}$, 
and hence $k\leq \frac{x(v-1)}{q}$. This yields $\frac{v-1}{q}+1\leq k\leq 
\frac{2(v-1)}{q}$ if $x=2$ and $\frac{v-1}{q_0} +1\leq k\leq \frac{v-1}{q_0}+
\frac{v-1}{q}$ if $x=q_0+1$.
\end{proof}

\section{Suzuki, Ree and rank 1 Unitary groups}\label{sec-rank1}

In this section we treat the 2-transitive actions of Lie type groups $G$ 
of rank 1 apart from the linear case which is handed in Subsection~\ref{sub-linrank1}.
Again we use the notation from Section~\ref{sect:organisation}: since 
$G$ is not 3-transitive we assume that $3\leq k\leq v-3$.  There is an 
infinite family of examples connected to the classical unitals in 
$\PG(2,q^2)$, (for information on these unitals see \cite{be,t}).

Let $q$ be a prime power and $V=\F_{q^2}^3$. The involutory automorphism  
$x\rightarrow x^q$ of $\F_{q^2}$ allows us to define a Hermitian form 
$\varphi:V\times V\rightarrow\F_{q^2}$ as follows: for $\x=(x_1,x_2,x_3)$ 
and $\y=(y_1,y_2,y_3)\in V$, 
$\varphi(\x, \y)=x_1\bar{y_3} + x_3\bar{y_1} + x_2\bar{y_2}$, where we write 
$\bar{a}:=a^q$ for $a\in\F_{q^2}$. 

\begin{example}\label{ex-u}{\rm
The subgroup $G:=\PGaU(3,q)$ of $\PGaL(3,q^2)$ 
preserving $\varphi$ acts faithfully and 2-transitively on the set
$\V$ of $v=q^3+1$ isotropic $1$-spaces $\la \x\ra$ of $V$ (that is, 
$\varphi(\x,\x)=0$). Each non-degenerate 2-space $L$ of $V$, relative to $\varphi$, 
contains exactly $q+1$ elements of $\V$, and we denote this $(q+1)$-subset of 
$\V$ by  $L\cap\V$. The code $\Ga\subset\binom{\V}{q+1}$ consisting of 
these $(q+1)$-subsets,  one for each non-degenerate 2-space $L$, is the 
classical unital. It is $G$-strongly incidence-transitive 
with minimum distance $\delta(\Ga)=q$.  
}
\end{example}

\begin{lemma}\label{lem-u}
The claims made about $\Ga$ in Example~{\rm\ref{ex-u}} are valid. 
\end{lemma}

\begin{proof}
Let $\ga=L\cap\V$ 
for some non-degenerate 2-space $L$.
We prove that $G_\ga$  is transitive on $\ga\times\bar\ga$, 
where $G=\GU(3,q)$ (acting with kernel a subgroup of scalars of order 
$(3,q+1)$).
Denote by $\e_i$ the standard basis vector with 1 in the $i$-entry and 
other entries 0. Then $\e_1, \e_3$ are isotropic while $\varphi(\e_2,\e_2)=1$. 
We take $\ga=L\cap\V$, for $L$ the $\varphi$-orthogonal complement of the 
non-isotropic vector $\e_2$. It is straightforward to compute that
\[
\ga =\{\la \e_1\ra, \la \e_3\ra\} \cup\{\la (x,0,1)\ra\,|\, x^{q-1}=1\}
\]
of size $k=|\ga|=q+1$. The fact that $\GU(3,q)$ is transitive on 
incident point-block pairs of the unital follows from Witt's theorem, 
and hence $G_\ga$ is transitive on $\ga$. By \cite[pp.248--250]{DM}, 
$G$ contains (modulo scalars) the following elements
\begin{equation}\label{eq-t}
t_{\alpha,\beta} = \left(
% use packages: array
\begin{array}{lll}
1 & -\bar\beta & \alpha \\ 
1 & 1 & \beta \\ 
0 & 0 & 1
                         \end{array} \right), \quad        
h_{\nu,\mu}= \left(
\begin{array}{lll}
\nu & 0 & 0 \\ 
0 & \mu & 0 \\ 
0 & 0 & \bar\nu^{-1}
                         \end{array}  \right)  
\end{equation}
for $\alpha,\beta,\nu,\mu\in\F_{q^2}$ such that $\alpha+
\bar\alpha+\beta\bar\beta=0$, $\nu\ne0, \mu\bar\mu=1$, 
and the stabiliser $G_\v$  in $G$ of $\v=\la \e_3\ra\in\ga$ 
consists of the $q^3(q^2-1)(q+1)$ products $h_{\nu,\mu}
t_{\alpha,\beta}$. A straightforward computation shows 
that $G_{\ga,\v}$ has order $q(q^2-1)(q+1)$, comprising 
those products with $\beta=0$. For $x\in\F_{q^2}$ such 
that $x+\bar x+1 = 0$ and $x\ne1$, the vector $(x,1,1)$ is 
isotropic, so $\u:=\la (x,1,1)\ra\in\V\setminus\ga =
\bar\ga$. The element  $h_{\nu,\mu}t_{\alpha,0}\in 
G_{\ga,\v}$ maps $(x,1,1)$ to $(x\nu+\mu,\mu,x\nu\alpha +\bar\nu^{-1})$, 
and hence fixes $\u$ if and only if $x\nu+\mu=x\mu$ and
$\mu=x\nu\alpha +\bar\nu^{-1}$; or equivalently,
$\nu=\mu(x-1)/x$ and $\alpha= (\mu-\bar\nu^{-1})/x\nu$
are determined by $x$ and $\mu$.
% and $\nu=\mu$ (and hence $\nu=\bar\nu^{-1}$). 
Thus the $G_{\ga,\v}$-orbit 
containing $\u$ has length $q(q^2-1)=|\bar\ga|$, whence $G_\ga$ 
is transitive on $\ga\times\bar\ga$ as claimed. Finally every two 
points of $\V$ lie in a unique codeword in $\Ga$, and since $G$ is 
2-transitive on $\V$ the largest intersection of distinct codewords 
is 1, so the minimum distance of $\Ga$ is $q$. 
\end{proof}

We now prove Theorem~\ref{thm-r1}, which deals with the 2-transitive groups 
$G$ of rank 1, that is, groups with socle $T(q)$ of degree $v=|\V|$
as in one of the lines of Table~\ref{tbl-rank1}.

\begin{center}
\begin{table}
\begin{tabular}{llccc}\hline
$T(q)$	&$q=p^a$		&$v$  &$|T(q)|$            &$|\Out(T)|$\\ \hline
$\Sz(q)$& $2^{2c+1}>2$	&$q^2+1$&$q^2(q^2+1)(q-1)$ &$a$\\
$\Ree(q)$& $3^{2c+1}>3$&$q^3+1$&$q^3(q^3+1)(q-1)$ &$a$\\
$\Ree(3)'$& $3$&$q^3+1$&$q^2(q^3+1)(q-1)$ &$3$\\
$\PSU(3,q)$&$q>2$        &$q^3+1$&$\frac{1}{d}q^3(q^3+1)(q^2-1)$&$2ad$\\ 
          &              &       &where $d=(3,q+1)$             &    \\ \hline
\end{tabular}
\caption{Groups for Theorem~\ref{thm-r1}} \label{tbl-rank1}                                                 
\end{table}   
\end{center}

%$(q,k)=(3,12)$ has $\de=6$

\smallskip\noindent
\emph{Proof of Theorem~\ref{thm-r1}.}\quad
Let $T\trianglelefteq G\leq\Aut(T)$ with $q=p^a$,  $T=T(q)$ and $v=|\V|$ as in one of the 
lines of Table~{\rm\ref{tbl-rank1}}.
We use the classification of the subgroups of $G$ in \cite{u,ree,sz}
for the Suzuki, Ree and unitary groups, respectively. 
Suppose that $\Ga\subset\binom{\V}{k}$ is $G$-strongly 
incidence transitive, and let $\ga\in\Ga$. Since 
$\Ga\ne\binom{\V}{k}$ and $G$ is 2-transitive on $\V$, we have
$3\leq k\leq v-3$. 
Then $G_\ga$ has two orbits in $\V$, each of size at least 3, 
and it follows that 
$G_\ga$ is not contained in a parabolic subgroup. 
When $G_\ga$ is contained in other maximal subgroups we use the fact that 
$k(v-k)$ divides $|G_\ga|$ and in particular that  $k(v-k)\leq |G_\ga|$.

If $T=\Sz(q)$, then by \cite{sz}, the non-parabolic maximal subgroups of 
$T$ have orders $2(q-1)$, or $4(q\pm r+1)$, or $|\Sz(q_0)|$, where $2q=
r^2$ and $q=q_0^b$ for an odd prime $b$. In each case $|G_\ga|\leq a
|T_\ga|<3(q^2-2)\leq k(v-k)$.
% so $G_\ga$ cannot be transitive on $\ga\times\bar\ga$.

Suppose next that $T=\Ree(q)$, with $q>3$, or $\Ree(3)'\cong\PSL(2,8)$. 
Then  by \cite{ree}, the non-parabolic maximal subgroups of $T$ have 
orders $6(q+1)$, or $2|\PSL(2,q)|$ (with $q>3$), or $6(q\pm r+1)$ 
(with  $3q=r^2$), or $|\Ree(q_0)|$ (with $q=q_0^b$ for an odd prime 
$b$). In each case $|G_\ga|\leq |\Out(T)|.|T_\ga|<3(q^3-2)\leq k(v-k)$.
% so $G_\ga$ cannot be transitive on $\ga\times\bar\ga$.  

Thus $T=\PSU(3,q)$ with $q>2$. We may assume that neither $\ga$ nor $\ov\ga$ is 
as in Example~\ref{ex-u}. 
Then $G_\ga$ acts irreducibly on the underlying space $V=V(3,q^2)$, 
so $G_\ga$ is contained in an irreducible maximal subgroup $H$ of $G$, 
and $H\cap T$ is contained in a maximal subgroup $M$ of $T$. 
The list of maximal subgroups of $T$
can be found in \cite[pp. xxx]{u}, and we consider them in turn.
First, however, we deal with the small cases where $q\in\{3,4,5\}$. For these groups,
 lists of maximal subgroups of $T$ are available in \cite{At}, and for some properties we rely on
computations in {\sf GAP}~\cite{GAP} kindly done for us by Max Neunh\"offer.

\smallskip\noindent
{\it Case: $q=3$.}\quad For $\v\in\V$ lying in a $G_\ga$-orbit
of length $\min\{k,v-k\}\leq v/2$, the subgroup $G_{\ga,\v}$ 
has an orbit of length $\max\{k,v-k\}\geq v/2=14$.
It follows that $T_\ga$ is not contained in the transitive 
maximal subgroup $\PSL(2,7)$, and hence $T_\ga=4^2:S_3$. 
A {\sf GAP} computation confirms that this subgroup gives rise to an 
example with $\ga$ or $\ov\ga$ of size $12$ and $\delta(\Ga)=6$, and for  
the transitivity condition we need $G=T.2$.
In this example, the codewords of size $12$ are the `bases' \cite[page 14]{At}.

\smallskip\noindent
{\it Case: $q=4$.}\quad 
Since $G_\ga$ is irreducible and $|G_\ga|$ is divisible by $k(65-k)$, 
it follows that $k=5$, $G= T.2$ or $T.4$, and $G_\ga\cap T.2=5^2:D_{12}$. However a 
 {\sf GAP} computation reveals that the subgroups $5^2:D_{12}$ and $5^2:(4\times S_3)$
both have orbit lengths 15 and 50 in $\V$,
and hence we get no example since $15\cdot 50$ does not divide $|G_\ga|$. 

\smallskip\noindent
{\it Case: $q=5$.}\quad  
Since $|G_\ga|$ is divisible by $k(126-k)$ 
it follows that $k=6$ and $G_\ga\cap T=M_{10}$, which has two orbits in $\V$.
However a {\sf GAP} computation shows that these orbit lengths are 36 and 90,
and  $36\cdot 90$ does not divide $|G_\ga|$. 

%\item[(iv)] Suppose that $q=7$. Since $|G_\ga|$ is divisible by $k(344-k)$ it follows that $k=8$ and $G_\ga\cap T=2(L_2(7)\times 4).2$ is the stabiliser of a non-isotropic point,  and $G=T$ or $T.2$. .[Possible example?]..,    

\smallskip\noindent

From now on we assume that $q\geq 7$. 

\smallskip\noindent
\emph{Case: $M$ preserves a direct decomposition of $V$.}\quad Then 
 $M$ is of type $(q+1)^2:S_3$,  
$M=H\cap T$, and $M$ has order $6(q+1)^2/(3,q+1)$. 
Let $ \v\in\V$ and note that $|T_{\v}|=q^3(q^2-1)/(3,q+1)$. Let 
\[
b:=(|M|,|T_{\v}|) =  \frac{q+1}{(3,q+1)} \,(6(q+1), q^3(q-1))
\]
and note that $|M_\v|$ divides $b$, so that the orbit length $|\v^M|$
is divisible by $|M|/b$.
We claim that $|\v^{M}|\geq \frac{q+1}{(2,q-1)}$ (so, since 
this holds for all $\v$ it implies that $k\geq  \frac{q+1}{(2,q-1)}$).
If $p=2$ then $b\leq 6(q+1)/(3,q+1)$ and so $|\v^{M}|\geq 
q+1$. Assume now that $p$ is odd. 
If $q\equiv 0$ or $1\pmod{3}$, then $b= 3(q+1)(2(q+1),\frac{q^3(q-1)}{3})\leq 12(q+1)$,
so $|\v^{M}|\geq \frac{q+1}{2}$. Finally if $q\equiv 2\pmod{3}$, 
then $b\leq \frac{q+1}{3}(2(q+1),q-1)\leq 4(q+1)/3$,
so $|\v^{M}|\geq \frac{3(q+1)}{2}$, and the claim is proved. 
Thus $k\geq   \frac{q+1}{(2,q-1)}$ and hence
\[
 \frac{q+1}{(2,q-1)}(q^3+1- \frac{q+1}{(2,q-1)})\leq k(q^3+1-k)\leq  |G_\ga|\leq (q+1)^2.6.2a 
\]
which implies that $q(q-1)$ is at most $12a$ if $q$ is even, or $24a$ if $q$ is odd. 
Since $q\geq7$, this means that $q=8$.
However if $q=8$ then $b=6$ and hence $k$ is divisible by 
$|M|/6=27$,  but then $k(q^3+1-k)$ does not divide $|G_\ga|$.   

\smallskip\noindent
\emph{Case: $M$ preserves an extension field structure on $V$.}\quad 
Here $M$ is of type $(q^2-q+1):3$.  However the cyclic group of 
order $q^2-q+1$ is semiregular on $\V$, and so, 
for  $\v\in\V$, $|G_{\ga,\v}|$ divides $6a$, which is 
less than $(q^3+1)/2$, so $G_\ga$ is not transitive on $\ga\times\bar\ga$.

\smallskip\noindent
\emph{Case: $M$ is a subfield subgroup.}\quad  
Suppose first that $q$ is odd and $M$ is of type ${\rm SO}(3,q)$. 
Then $|H|\leq q(q^2-1)2a$ (see \cite[Proposition 4.5.5]{KL}). 
Also $H$ is intransitive on $\V$ and hence $G_\ga = H$. 
Modulo scalars we can take $H\cap T$ to be the subgroup of matrices with 
entries in $\F_q$, so in particular $H$ contains the subgroup $H_0$ consisting of the 
$q$ matrices $t_{\alpha,\beta}$ in (\ref{eq-t}) with $\alpha,\beta\in\F_q$ and 
$2\alpha + \beta^2=0$. Consider the points $\v=\la \e_1\ra$ and $\u=\la (x,1,1)\ra$ 
defined in the proof of Lemma~\ref{lem-u}, where here we choose $x\in\F_{q^2}\setminus\F_q$ 
as well as satisfying $x+\bar x+1=0$. With this choice of $x$, the points $\v$ and $\u$ 
lie in different $H$-orbits (since $x\not\in\F_q$), and a straightforward 
calculation shows that each of the $H_0$-orbits 
containing $\v$ and $\u$ has length $q$. Thus $k\geq q$ and hence
\[
 q^2(q^2-1)<q(q^3+1-q)\leq k(q^3+1-k)\leq|G_\ga|\leq q(q^2-1).2a.
\]
This implies that $q<2a$, which is a contradiction.

Now suppose that $M$ is of type $\SU(3,q_0)$ with $q=q_0^r$ and $r$ 
an odd prime. Then arguing as above we have $G_\ga=H$ and, modulo scalars, we may take 
$H\cap T$ to be the subgroup of matrices with entries in $\F_{q_0^2}$, so in 
particular $H$ contains the subgroup $H_0$ consisting of the $q_0^3$ matrices 
$t_{\alpha,\beta}$ in (\ref{eq-t}) with $\alpha,\beta\in\F_{q_0^2}$ and $\alpha + 
\bar\alpha + \bar\beta \beta=0$. The points $\v$ and $\u$ lie in different $H$-orbits, 
where this time we take the scalar $x\in\F_{q^2}\setminus\F_{q_0^2}$, and the $H_0$-orbits 
containing these two points both have length $q_0^3$. Thus $k\geq q_0^3$ and so in this case, 
since $q\geq q_0^3$,
\[
\frac{q_0^{3r+3}}{2}< q_0^3(q^3-q_0^3)<|G_\ga|\leq q_0^3(q_0^3+1)(q_0^2-1).2a < q_0^8.2a
\]
and hence $4a> q_0^{3r-5}>q_0^{r}=q$ and we have a contradiction. 

\smallskip%\noindent
For each of the remaining groups $M$, we have $q=p\geq 7$.

\smallskip\noindent
\emph{Case: $M$ is of symplectic type.}\quad The group $M$ corresponds to a subgroup  
$3^{1+2}:Q_8.\frac{(q+1,9)}{3}$ of $\SU(3,q)$ and here $q\geq11$. 
The order $|G_\ga|$ is at most $1296$ which is less than $k(q^3+1-k)$.

\smallskip\noindent
\emph{Case: $M=\PSL(2,7)$ with $q\equiv 3,5,6\pmod{7}$.}\quad We have 
$3(q^3-2)\leq 168\cdot 2$ which is a contradiction for $q\geq7$.

\smallskip\noindent
\emph{Case: $M=A_6$ with $q\equiv 11, 14\pmod{15}$.}\quad We have 
$3(q^3-2)\leq 360\cdot 2$ which is a contradiction for $q\geq7$.

\smallskip
This completes the proof of Theorem~\ref{thm-r1}.

\end{document}